\documentclass[11pt]{amsart}
\usepackage[utf8]{inputenc}
\usepackage{a4wide}
\usepackage{amsthm,amssymb,amsmath,mathrsfs,mathtools,eucal,cancel}
\usepackage{xcolor}
\usepackage{bbm,stix}
\usepackage{enumitem}
\usepackage{wrapfig}
\usepackage{dsfont}
\usepackage{fullpage}
\usepackage{nccmath} 
\usepackage{booktabs}
\usepackage{longtable}
\usepackage{soul} 

\usepackage{tikz, pgfplots, transparent}
\usepackage{tikz-cd}
\pgfplotsset{compat=1.17}
\usetikzlibrary{arrows, arrows.meta, automata, calc, chains, decorations.markings, decorations.pathmorphing, decorations.pathreplacing, fadings, positioning, shapes}

\usepackage{cite}

\usepackage{graphicx,caption,subcaption}
\usepackage[bookmarks=false]{hyperref}
\hypersetup{
  colorlinks   = true, 
  urlcolor     = black, 
  linkcolor    = blue, 
  citecolor   = blue 
}

\theoremstyle{definition}
\newtheorem{defi}{Definition}[section]

\theoremstyle{plain}
\newtheorem{theorem}[defi]{Theorem}
\newtheorem{lem}[defi]{Lemma}
\newtheorem{cor}[defi]{Corollary}
\newtheorem{prop}[defi]{Proposition}
\theoremstyle{remark}
\newtheorem{rem/}{Remark}

\newenvironment{rem}
  {%
   \pushQED{\qed}\begin{rem/}}
  {\popQED\end{rem/}}

\newcommand*{\Eq}{Eq.~}

\newcommand*{\di}{\mathrm{d}}
\newcommand{\bbE}{\mathbb{E}}
\newcommand{\I}{\mathbb{I}}
\newcommand{\bbp}{\mathbb{p}}
\newcommand{\bbP}{\mathbb{P}}
\newcommand{\bbr}{\mathbb{r}}
\newcommand{\R}{\mathbb{R}}
\newcommand{\bbS}{\mathbb{S}}
\newcommand{\T}{\mathbb{T}}
\newcommand{\V}{\mathbb{V}}
\newcommand{\Z}{\mathbb{Z}}
\newcommand{\eps}{\varepsilon}

\newcommand{\calC}{\mathcal{C}}

\newcommand{\calF}{\mathcal{F}}
\newcommand{\calG}{\mathcal{G}}

\newcommand{\calL}{\mathcal{L}}
\newcommand{\calM}{\mathcal{M}}
\newcommand{\calP}{\mathcal{P}}
\newcommand{\scrR}{\mathscr{R}}

\newcommand{\scrH}{\mathscr{H}} 
\newcommand{\scrL}{\mathscr{L}} 
\newcommand{\scrI}{\mathscr{I}}	
\newcommand{\scrJ}{\mathscr{J}}	
\newcommand{\scrS}{\mathscr{S}}	
\newcommand{\stat}{\pi} 
\newcommand{\FI}{\mathcal{D}} 
\newcommand{\Mol}{\mathrm{M}} 
\renewcommand*{\mid}{\,\vert\,}

\DeclareMathOperator{\Div}{div}
\DeclareMathOperator{\dnabla}{\overline\nabla}
\DeclareMathOperator{\ddiv}{\overline{div}}
\DeclareMathOperator{\dpartial}{\overline\partial}
\DeclareMathOperator{\mnabla}{\widetilde{\nabla}}
\DeclareMathOperator{\mdiv}{\widetilde{div}}

\DeclareMathOperator\CE{CE} 
\DeclareMathOperator\ME{ME} 
\DeclareMathOperator\PCE{PCE} 

\DeclareMathOperator\Ent{Ent} 
\DeclareMathOperator\TV{TV} 
\DeclareMathOperator\BL{BL} 

\makeatletter
\DeclareFontFamily{OMX}{MnSymbolE}{}
\DeclareSymbolFont{MnLargeSymbols}{OMX}{MnSymbolE}{m}{n}
\SetSymbolFont{MnLargeSymbols}{bold}{OMX}{MnSymbolE}{b}{n}
\DeclareFontShape{OMX}{MnSymbolE}{m}{n}{
    <-6>  MnSymbolE5
   <6-7>  MnSymbolE6
   <7-8>  MnSymbolE7
   <8-9>  MnSymbolE8
   <9-10> MnSymbolE9
  <10-12> MnSymbolE10
  <12->   MnSymbolE12
}{}
\DeclareFontShape{OMX}{MnSymbolE}{b}{n}{
    <-6>  MnSymbolE-Bold5
   <6-7>  MnSymbolE-Bold6
   <7-8>  MnSymbolE-Bold7
   <8-9>  MnSymbolE-Bold8
   <9-10> MnSymbolE-Bold9
  <10-12> MnSymbolE-Bold10
  <12->   MnSymbolE-Bold12
}{}
\let\llangle\@undefined
\let\rrangle\@undefined
\DeclareMathDelimiter{\llangle}{\mathopen}%
                     {MnLargeSymbols}{'164}{MnLargeSymbols}{'164}
\DeclareMathDelimiter{\rrangle}{\mathclose}%
                     {MnLargeSymbols}{'171}{MnLargeSymbols}{'171}
\makeatother

\definecolor{bananayellow}{rgb}{1.0, 0.88, 0.21}
\newcommand*\UScomment[1]{}

\def\parab(#1, #2){exp(-#2**2/(2*#1))/sqrt(2*pi*#1)}
\def\parstep(#1, #2){(1 + erf(#2/sqrt(4*#1)))/2}
\def\hyper(#1, #2){.25*exp(-.5*#1)*(besi0(.5*real(sqrt(#1**2-#2**2)))+#1*besi1(.5*real(sqrt(#1**2-#2**2)))/real(sqrt(#1**2-#2**2)))}

\numberwithin{equation}{section}

\title{Fourier-Cattaneo equation: stochastic origin, variational formulation, and asymptotic limits}

\def\@setthanks{\vspace{-\baselineskip}\def\thanks##1{\@par##1\@addpunct.}\thankses}
\makeatother

\author{Alberto Montefusco, Upanshu Sharma, Oliver Tse}
\date{\today}

\allowdisplaybreaks

\begin{document}

\begin{abstract}
    We introduce a variational structure for the Fourier-Cattaneo (FC) system which is a second-order hyperbolic system. This variational structure is inspired by the large-deviation rate functional for the Kac process which is closely linked to the FC system. Using this variational formulation we introduce appropriate solution concepts for the FC equation and prove an a priori estimate which connects this variational structure to an appropriate Lyapunov function and Fisher information—the so-called FIR inequality. Finally, we use this formulation and estimate to study the diffusive and hyperbolic limits for the FC system. 
\end{abstract}
\maketitle

\tableofcontents

\section{Introduction}

Since the pioneering works of Onsager and Machlup \cite{OM53}, it has been known that a force-flux constitutive law at the macroscopic (coarser) scale is the manifestation of averaging effects that one observes in the passage from a microscopic (finer) level to a macroscopic (coarser) level of description. Over the last decade, this intuition has been made precise via the connection between underlying stochastic particle systems and macroscopic diffusion equations using the language of large deviations~\cite{AdamsDirrPeletierZimmer11,AdamsDirrPeletierZimmer13,MPR14}. Large-deviation theory lends a natural \emph{variational structure}---which is tightly linked to gradient-flow theory---to the diffusion equation (and other related parabolic systems), thereby making the notion of a force-flux constitutive relation precise. Yet, to the best of the authors' knowledge, no such link has either been established or investigated for hyperbolic systems.

The first aim of this work is, therefore, to provide a starting point for the development of \emph{variational structures} for hyperbolic systems. For this, we consider the so-called \emph{hyperbolic heat equation} introduced by Cattaneo~\cite{cC48} (see~\cite{JosephPreziosi89} for a detailed survey)  
\begin{equation}\label{Cattaneo}
    \tau \partial_t^2 \rho + \partial_t \rho = \Div(\alpha \, \nabla \rho ) \, .
\end{equation}
The relaxation time $\tau$ quantifies the time the system takes to respond to a force, and can be seen by writing~\eqref{Cattaneo} as a system of first-order equations, often called the  \emph{Fourier-Cattaneo} (FC) \emph{system}
\begin{subequations}\label{Cattaneo2}
    \begin{align}
        \partial_t \rho &= - \Div \omega \, , \\
        \partial_t \omega &= - \frac{1}{\tau} \bigl(\omega + \alpha \, \nabla \rho\bigr) \, .
    \end{align}
\end{subequations}
Historically, Cattaneo introduced this model in the context of heat conduction with a modified Fourier law to overcome the problem of an infinite speed of propagation (see Figure~\ref{fig:delta}). It should be noted that several models with finite speed of propagation exist in the literature, for instance in~\cite{RosenauKamin82,DebbaschMallickRivet97,DunkelHanggi09}. However, all of these are purely parabolic and in this article we will focus on the hyperbolic FC system~\eqref{Cattaneo2}.

We will derive a variational structure for the FC system using a stochastic system introduced in \cite{Furth1920,Taylor1921,Goldstein1951}, and subsequently studied by Kac~\cite{mK74} and McKean~\cite{McKean1967}. While the resulting variational structure does reveal new insights into the FC equation, our procedure has limitations in higher dimensions, as we discuss below; therefore, in this article, we will restrict ourselves to the FC system in one spatial dimension. We envision that this first study of variational structures for such equations will provide an alternate physically-motivated viewpoint to hyperbolic equations and widen the scope of techniques developed to study gradient flows. 

The second aim of this work is to use this variational structure to rigorously analyse the asymptotic behaviour of the FC system in two limiting regimes, namely, the diffusive limit (Section~\ref{subsec:parabolic}) and the hyperbolic limit (Section~\ref{subsec:hyperbolic}) using evolutionary convergence and the recently introduced FIR inequality \cite{HPST20}. Formal and rigorous diffusive and hyperbolic limits for kinetic models using various techniques can be found e.g.\ in \cite{LionsToscani1997,BDDLTV2003,Eftimie2012,BellouquidChouhad2016}. Although asymptotic limits of the FC system have been studied in the past, for instance, via Chapman-Enskog expansions (e.g.~in \cite{McKean1967,OthmerHillen00}), our method offers an alternative approach to proving asymptotic limits using variational techniques and with minimal assumptions on the initial data. 

\begin{figure}[t!]
\includegraphics{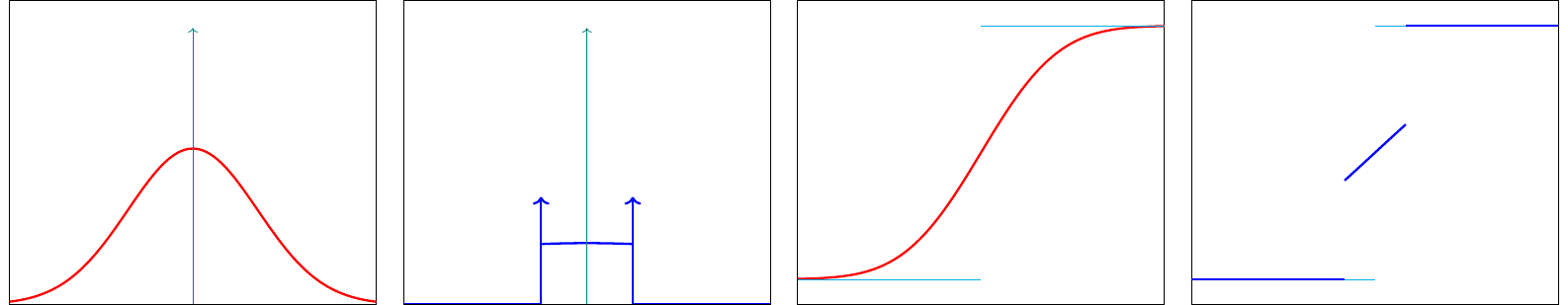}
    \caption{Plot of the fundamental solution to the diffusion equation $\partial_t \rho = \Div(\alpha \, \nabla \rho)$ in red and hyperbolic heat equation~\eqref{Cattaneo} in blue for two initial conditions $\rho_0=\delta_0,\mathbb 1_{[0,\infty)}$ (in green and cyan respectively) in $\R$. The diffusion equation clearly manifests an infinite speed of propagation with a smooth profile and unbounded support, while the support of the hyperbolic equation evolves discontinuously with a finite speed. }
    \label{fig:delta}
\end{figure}

\subsection{Stochastic model for FC equation}\label{sec:Poisson-Kac}
%
Parabolic equations often arise as hydrodynamic limits of (possibly interacting) particle systems (see \cite{KL99} and references therein). For instance, the diffusion equation $\partial_t \rho = \Div(\alpha \, \nabla \rho)$ can either be viewed as the hydrodynamic limit of independent Brownian motions on a continuous state space or of an exclusion process on a discrete lattice. In comparison, the literature on stochastic particle systems for second-order hyperbolic equations is far less developed (cf.~\cite{HillenHadeler05} and references therein). In \cite{mK74}, Kac studied a simple jump-process model that he formally connected to the FC system and the closely related telegrapher's equation. We now briefly describe this particle system and its connection to the FC system~\eqref{Cattaneo}. In \appendixname~\ref{sec:LD}, we give a heuristic motivation for the large deviations, which provides the basis for the variational structure.

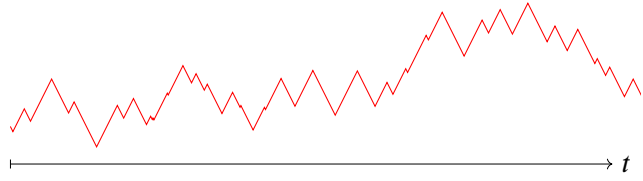
\begin{figure}[ht]
    \begin{tikzpicture}
        \draw[|->, ultra thin] (0, -.5) -- (8, -.5) node[at end, right] {$t$}; \pgfmathsetseed{1234} \useasboundingbox (0,-2) rectangle (8,2);
        \clip (current bounding box.north east) -- (current bounding box.north west) -- (current bounding box.south west) -- (current bounding box.south east) -- cycle;
        \draw[red] (0,0) \foreach \x in {1,...,100} {\foreach \y in {-1,1} {-- ++($rnd*(.3,\y*.6)$)}};
    \end{tikzpicture}
    \vspace{-1cm}
    \caption{A sample path of the Kac process with speed~$V = 2$ and switching rate~$\lambda=\frac12$.}
    \label{fig:Kac}
\end{figure}

Consider a particle moving in a one-dimensional torus $\T$ with a constant speed $V>0$ that may switch its direction according to a time-homogeneous Poisson process with the rate $\lambda$---see Figure~\ref{fig:Kac} for a sample path of this process. This process, called a \emph{Kac process} in this article, is a piecewise deterministic Markov process on $\Omega_V \coloneqq \T \times \{-V, V\}$. We now consider $N$ independent copies of this process labeled by the pair of position and velocity $(x^i_t,v^i_t)\in\Omega_V$. Throughout this article, we use the subscript for the evaluation at time $t$. 
Classical results~\cite[Section 11.4]{Dudley18} state that the \emph{empirical measure} 
\begin{equation}\label{empirical_process}
    \sigma_t^N \coloneqq \frac1N \sum\limits_{i=1}^N \delta_{(x^i_t, v^i_t)} \, ,
\end{equation}
converges almost surely, in the many-particle limit $N\rightarrow\infty$, to the measure-valued evolution 
\begin{equation}\label{marginals}
    \partial_t \sigma + v \, \partial_x \sigma = \lambda \, (\iota_\sharp\sigma - \sigma) \, ,\qquad\text{with\;  $\sigma_t=\mathrm{law}(x^i_t,v^i_t)$}\,,
\end{equation}
where $\iota_\sharp\sigma$ is the push-forward of $\sigma$  under the velocity-reversal map $\iota(x,v)=(x,-v)$, and hence $(\iota_\sharp\sigma)(dx,v)=\sigma(dx,-v)$. Henceforth, we will refer to~\eqref{marginals} as the \emph{Kac equation}. 
To illustrate the connection of~\eqref{marginals} to the FC system, we introduce the \emph{density} $\rho$ and the \emph{flux} $\omega$ as
\begin{equation}\label{eq:intro-biject}
    \rho(\di x) \coloneqq \sum\limits_{v \in \{-V, V\}} \! \sigma(\di x, v) \, , \quad 
    \omega(\di x) \coloneqq \sum\limits_{v \in \{-V, V\}}\! v \, \sigma(\di x, v) \, .
\end{equation}
The mapping $\sigma_t\mapsto (\rho_t,\omega_t)$ is in fact a bijection, and the inverse mapping is given by
\begin{equation}\label{bijection}
    \sigma(\di x, v) = \frac12\Bigl( \rho(\di x) + \frac{\omega(\di x)}{v} \Bigr) \, .
\end{equation}
It is easily checked that the density-flux pair formally evolves according to
\begin{subequations}\label{damped_hyperbolic}
    \begin{align}
        \partial_t \rho &= - \partial_x \omega \, , \label{density} \\
        \partial_t \omega &= - V^2 \partial_x \rho - 2 \lambda \, \omega \, , \label{flux}
    \end{align}
\end{subequations}
which is the FC system~\eqref{Cattaneo2} in one-dimension with $\alpha = V^2/(2\lambda)$ and $\tau = 1/(2\lambda)$.

As stated above, we focus on the one-dimensional FC system~\eqref{damped_hyperbolic}, a rather restrictive setting which is due to its connection to the Kac process. Consider, instead, a (Kac-type) particle moving in two dimensions, with the velocity switching randomly between four possibilities $\{(-V, 0), (V, 0), (0, -V), (0, V)\} \eqqcolon \V$ with the rate $\lambda$. The corresponding law of the process, which is also the limit of the empirical measure as above but now in two-dimensions, reads 
\begin{equation*}
    \partial_t\sigma(x, v) + v \cdot \nabla_x\sigma(x, v) = \lambda \sum\limits_{u \in \V \backslash \{v\}} \! \bigl( \sigma(x, u) - \sigma(x, v) \bigr) \, .
\end{equation*}
The evolution for the density-flux pair in this two-dimensional setting is given by 
\begin{align*}
    \partial_t\rho &= -\Div\omega \, , \\
    \partial_t \omega &= - \Div\Bigl[ \sum\limits_{v \in \V} (v\otimes v)\, \sigma(x, v) \Bigr] - 4 \lambda \, \omega \, .
\end{align*}
The evolution for the flux $\omega$ is not closed since it requires the second-order moment of $\sigma$ in $v$. Clearly, the evolution of this second moment requires information on further higher-order moments, which leads to an infinite set of equations. This is not an issue in the one-dimensional setting since, for $v\in\{-V,V\}$, we have $v^2 = V^2$, and therefore the second moment in the flux evolution reduces to the zeroth moment~$\rho$. In other words, the relation between $\sigma$ and the pair $(\rho, \omega)$ is a bijection only in one dimension. This is a strong limitation of the Kac process and it is unclear how to construct higher-dimensional analogues which circumvent this issue. The appearance of an infinite chain of moments is a typical phenomenon in statistical mechanics and indicates that a few moments are not sufficient to describe the system unless we enforce an artificial closure \cite{Hillen2002,Kuehn2016,KST2014}, which only acts as an approximation to the macroscopic FC system, or perform hyperbolic scaling limits \cite{Perthame2004,FLP2005,DolakSchmeiser2005}.

\subsection{Outline of results}\label{subsec:main-results} The first half of the paper is devoted to developing the variational structure for the FC system \eqref{damped_hyperbolic} (Section~\ref{sec:variational}) and deducing the implications of the structure (Section~\ref{sec:FIR}). We begin by introducing a variational structure for the Kac equation~\eqref{marginals}.

Define the functional $\scrI\colon C([0,T];\calP(\Omega_V))\times \calM([0,T];\calM(\Omega_V;\R^2))\rightarrow[0,+\infty]$, with $j=(j^1,j^2)$ as
\begin{equation}\label{RF}
    \scrI(\sigma, j) =
    \begin{dcases*}
        \int_0^T \Ent(j^2_t \mid \lambda \, \sigma_t)\,dt & if $\partial_t\sigma + \ddiv j = 0$ and $j^1 = v \, \sigma$, \\
        + \infty & otherwise.
    \end{dcases*}
\end{equation}
Here $\Ent(\cdot|\cdot)$ is the relative entropy of measures and $\ddiv$ is the divergence operator; they are defined in \eqref{def:relEnt} and \eqref{def:Omega-der} respectively (cf.~also Definition~\ref{def:CE} for the definition of the continuity equation). This functional, which we call the \emph{rate functional} since it is inspired by the large deviations of the Kac process (see \appendixname~\ref{sec:LD} for details), is a variational formulation for the Kac equation~\eqref{marginals} in the sense that
\begin{equation*}
    \scrI\geq 0 \;\; \text{ and } \;\; \scrI(\sigma,j)=0 \quad\Longleftrightarrow\quad \sigma \text{ solves } \eqref{marginals}.
\end{equation*}
Using the bijective mapping~\eqref{bijection}, we construct an equivalent variational formulation for the FC system~\eqref{damped_hyperbolic} (in the sense as above) via the relation
\begin{equation}\label{varstr}
    \scrJ(\rho, \omega) = \inf
    \bigl\{ \scrI(\sigma, j): \partial_t\sigma + \ddiv j =0 \, , \; \sigma = \rho + \omega/v \bigr\}\,.
\end{equation}
Both $\scrI$ and $\scrJ$ have a logarithmic structure inherited from the relative entropy, which is in sharp contrast to quadratic structures for related second-order hyperbolic systems \cite[Section~5.4]{PKG18}. 

The variational structure provided by \eqref{RF} allows us to establish the so-called FIR inequality
in Section~\ref{sec:FIR}:
\begin{equation}\label{eq:intro_FIR}
\Ent(\sigma_t\mid\stat) + \lambda \int_0^t \FI(\sigma_r\mid\stat)\,dr \leq \Ent(\sigma_0\mid\stat) + \scrI(\sigma,j) \, ,\qquad\text{for every $t\in[0,T]$}\,,
\end{equation}
which relates the \emph{free energy} $\Ent(\sigma\mid\stat)$, the \emph{Fisher information} $\FI(\sigma\mid\stat)$ (see \eqref{def:FI} for its definition), and the rate function $\scrI(\sigma,j)$ for any pair $(\sigma,j)$ with $\scrI(\sigma,j)<\infty$. Similar estimates have been discussed in recent years for a variety of systems~\cite{DLPS17,DLPSS18,HPST20,PRS21}.  

By projecting the FIR inequality onto the density-flux pair $(\rho,\omega)$, we obtain 
\begin{equation*}
    \Ent(\rho_t\mid\mathcal L_{\T}) + \frac{1}{2\alpha} \int_0^t\left\|\frac{d\omega_r}{d\rho_r}\right\|_{L^2(\T,\rho_r)}^2 \,dr \leq \Ent(\sigma_0\mid\stat) + \scrI(\sigma,j) \, ,\qquad\text{for every $t\in[0,T]$}\,,
\end{equation*}
where $\mathcal L_\T$ is the Lebesgue measure on the torus $\T$. This inequality is the main ingredient in establishing compactness for density-flux-pair sequences in the later part of the paper.

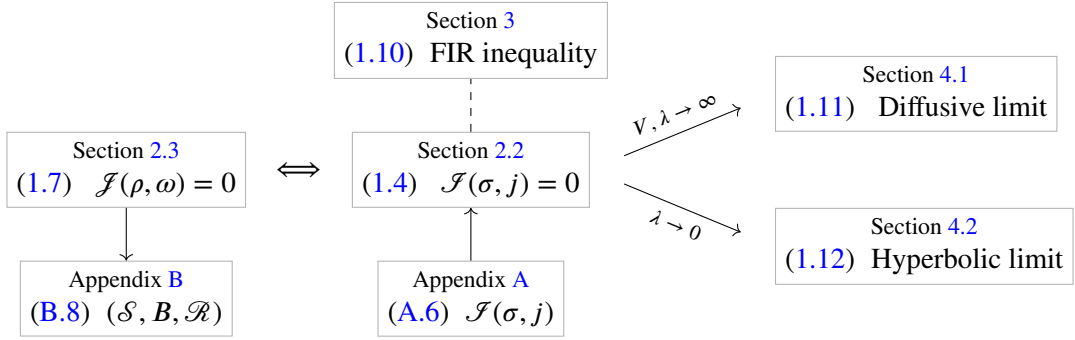
\begin{figure}[ht]
 \begin{tikzpicture}
  \tikzstyle{r}=[rectangle, draw=black!30, align=center, anchor=west]
  \node[r] (A) {\footnotesize{Section~\ref{sec:Var}}\\\eqref{damped_hyperbolic}\;\; $\scrJ(\rho,\omega)=0$};
  \node[r, right=40 pt of A] (B) {\footnotesize{Section~\ref{subsec-Kac-var}}\\\eqref{marginals}\;\; $\scrI(\sigma,j)=0$};
  \node (C) at ($(A)!0.5!(B)$) {$\Longleftrightarrow$};
  \node[r, right=70 pt of B.north east, anchor=south west] (D) {\footnotesize{Section~\ref{subsec:parabolic}} \\ \eqref{eq:intro_heat}\;\; Diffusive limit};
  \draw[->, shorten <=14pt, shorten >=14pt] (B.east) -- (D.west) node[midway, above, sloped] {\scriptsize{$V,\lambda\to\infty$}};
  \node[r, right=70 pt of B.south east, anchor=north west] (H) {\footnotesize{Section~\ref{subsec:hyperbolic}} \\ \eqref{eq:intro_wave}\;\;Hyperbolic limit};
  \draw[->, shorten <=14pt, shorten >=14pt] (B.east) -- (H.west) node[midway, below, sloped] {\scriptsize{$\lambda\to0$}};
  \node[r, above=20 pt of B.north] (G) {\footnotesize{Section~\ref{sec:FIR}} \\ \eqref{eq:intro_FIR}\;\;FIR inequality};
  \draw[dashed] (B.north) -- (G.south);
  \node[r, below=20 pt of B.south] (LD) {\footnotesize Appendix~\ref{sec:LD}\\\eqref{RateFunction}\;\;$\scrI(\sigma, j)$};
  \draw[->] (LD) -- (B);
  \node[r, below=20 pt of A.south] (G) {\footnotesize Appendix~\ref{sec:preGEN}\\\eqref{preGEN}\;\;$(\scrS,B,\scrR)$};
  \draw[->] (A) -- (G);
 \end{tikzpicture}
 \caption{Schematic outline of the article. The variational structure $\scrI(\sigma, j)$ and the corresponding solution concepts (as zeros of this structure) for the Kac equation~\eqref{marginals} and the FC system~\eqref{damped_hyperbolic} are introduced in Sections~\ref{subsec-Kac-var} and~\ref{sec:Var} respectively. Section~\ref{sec:Var} also shows that these two equations are equivalent. The FIR inequality for the Kac equation is introduced in Section~\ref{sec:FIR}. The asymptotic parabolic and hyperbolic limits are discussed in Sections~\ref{subsec:parabolic} and~\ref{subsec:hyperbolic} respectively. In \appendixname~\ref{sec:LD} we 
 heuristically motivate the variational structure~$\scrI(\sigma, j)$ from large deviations, and in \appendixname~\ref{sec:preGEN} we derive a pre-GENERIC structure ($\scrS,B,\scrR$) for the FC system from the variational structure~$\scrJ(\rho, \omega)$.}
\end{figure}

In the second half of the paper, we use the aforementioned variational structures to study two asymptotic limits of~\eqref{marginals} (and equivalently of the FC system~\eqref{damped_hyperbolic}). We now briefly discuss both of these limits and outline the variational technique used to study these limits. 

The first limit is a \emph{diffusive limit} where both $V$ and $\lambda$ grow to infinity, such that~$\alpha \coloneqq V^2/(2 \lambda)$ stays fixed. At the level of the underlying (stochastic) Kac process, this corresponds to the situation where both the speed $V$ of the particle and the switching rate $\lambda$ of the velocity become large. This is reminiscent of the usual diffusive/parabolic scaling for stochastic particle systems that leads to diffusive equations in the limit. This is exactly observed in our setting, with the limiting system given by the parabolic diffusion equation
\begin{subequations}\label{eq:intro_heat}
    \begin{align}
        \partial_t \rho + \partial_x \omega &= 0 \, , \\
        \omega &= - \alpha \, \partial_x \rho \, .
    \end{align}
\end{subequations}
Such a limit is also observed by formally passing $\tau\rightarrow 0$ in the original FC system~\eqref{Cattaneo2}.

The damped hyperbolic system~\eqref{damped_hyperbolic} presents another interesting limit when the switching rate~$\lambda$ vanishes while the speed~$V$ remains constant. In this case, we expect that any smoothing effect is completely removed: the initial mass is transported across space without being distorted. From \eqref{damped_hyperbolic}, we may directly infer the limit equations
\begin{subequations}\label{eq:intro_wave}
\begin{align}
    \partial_t \rho &= - \partial_x \omega \, , \\
    \partial_t \omega &= -V^2 \partial_x \rho \,.
\end{align}
\end{subequations}
This is a wave equation with speed of propagation~$V$.

We use a variational technique developed in~\cite{DLPS17} to study these limits. To illustrate the idea, assume that the family of pairs $(\sigma^\eps,j^\eps)$ is a variational (equivalently, weak) solution of the Kac equation~\eqref{eq:strong-form}, parameterized by some scale-separation parameter $\eps>0$; $\eps=1/V, \, \lambda$ in the diffusive (with $\lambda$ rewritten in terms of $V,\alpha$) and hyperbolic limit respectively. Our aim is to establish the behaviour of the system as $\eps\rightarrow 0$.
Since the solutions are characterized by the rate functional $\scrI^\eps$ via $\scrI^\eps(\sigma^\eps,j^\eps)=0$, we establish the asymptotic behaviour as $\eps\rightarrow 0$ by answering the following two questions:
\begin{enumerate}
    \item \emph{Compactness:} Do solutions of $\scrI^\eps(\sigma^\eps,j^\eps)=0$ have sufficient compactness properties allowing  one to extract a converging subsequence in a suitable topology $\mathbb{t}$?
    \item \emph{Liminf inequality:} Is there a limit functional $\bar\scrI\geq 0$ such that 
    \begin{equation*}
        (\sigma^\eps,j^\eps) \overset{\mathbb{t}}{\longrightarrow} (\bar\sigma,\bar \jmath) \ \Longrightarrow \ \liminf_{\eps\searrow 0} \scrI^\eps(\sigma^\eps,j^\eps)\geq \bar\scrI(\bar\sigma,\bar
        \jmath)? 
    \end{equation*}
    And if so, does one have the equivalence
    \begin{equation*}
       \bar\scrI(\bar\sigma,\bar \jmath) = 0 \quad \Longleftrightarrow \quad (\bar\sigma,\bar \jmath) \text{ solves a well-posed limit dynamics?}
    \end{equation*}
\end{enumerate}
We answer both these questions for approximate solutions, i.e., for pairs $(\sigma^\eps,j^\eps)$ having finite rate function $\scrI^\eps(\sigma^\eps,j^\eps)<\infty$ and with well-prepared initial data---note that the right hand side of the FIR inequality \eqref{eq:intro_FIR} corresponds to exactly these requirements. The asymptotic behaviour of the exact solutions $\scrI^\eps(\sigma^\eps,j^\eps)=0$ is a special case of our analysis. 
The proofs of both these steps in the variational technique crucially hinge on a dual formulation of the rate functionals (Section~\ref{subsec-Kac-var}). 

In the diffusive limit, the limit functional $\bar\scrI$~\eqref{def:LimFun-diff} turns out to be a reformulation of the Wasserstein gradient structure for the diffusion equation \eqref{eq:intro_heat} (see Remark~\ref{rem:Wasserstein}), suggesting that the corresponding limit pairs $(\bar\rho,\bar\jmath)$ are Wasserstein-gradient-flow solutions of the diffusion equation. At the stochastic particle-system level, this limit can be seen as a transformation of Poissonian noise to Brownian noise as reflected in the quadratic-Wasserstein limiting functional.
In the hyperbolic limit, instead, the limit functional $\bar\scrI$~\eqref{eq:wave-var-form} trivializes, i.e., it takes the value zero for pairs $(\rho,\omega)$ satisfying \eqref{eq:intro_wave} and $+\infty$ otherwise. While in stark contrast with the diffusive limit, it is consistent with the variational structure for the Kac equation when $\lambda\to 0$, suggesting the `deterministic' behavior of particle trajectories. At the stochastic particle-system level, this limit can be seen as a complete removal of randomness which leads to the trivial limiting functional.

\subsubsection*{\textbf{Novelty.}} Considerable literature has been devoted to the study of variational structures for gradient flows in the last two decades and exploiting them to study asymptotic limits \cite{JKO98,Otto01,GiacomelliOtto01,SandierSerfaty04,CarrilloMcCannVillani06,AmbrosioGigliSavare08,PortegiesPeletier10,Mielke14TR,ArnrichMielkePeletierSavareVeneroni12}. In recent years, connections with underlying particle systems via large deviations have been exploited to expand this class to systems with additional non-dissipative effects, albeit mostly for diffusive systems~\cite{AdamsDirrPeletierZimmer11,AdamsDirrPeletierZimmer13,MPR14,DuongPeletierZimmer13,DLPS17,HPST20,PRST22}. In this work, we push variational structures towards hyperbolic equations via large deviations. The evolution of the flux in~\eqref{damped_hyperbolic} depends on the flux itself, which makes our system and the corresponding analysis different from the latter literature on non-dissipative systems where the evolution of the flux only depends on the density, effectively making the flux a dummy variable at the macroscopic level (see Section~\ref{sec:Dis} for a discussion).  

It should be noted that the well-posedness of the FC system in arbitrary dimensions (and many other related models) can be established using classical techniques for hyperbolic equations. Furthermore, the study of the asymptotic limits for the FC system is a classical problem that has been discussed by Kac~\cite{mK74} and others~\cite{OthmerHillen00,GBC16,GBC17d}. Since we are interested in the FC system as arising from a stochastic system, we consider it as a measure-valued evolution, in contrast to the classical hyperbolic framework. Consequently, our variational solution-concepts also differ from the classical literature. Furthermore, we study the asymptotic limits of the FC system via the convergence of the associated variational structures (following ideas in~\cite{DLPS17,HPST20}), which corresponds to the convergence of the FC system and fluctuations around it and requires minimal conditions on the initial data.

\subsection{Summary of the notation}

\begin{center}
\begin{longtable}{@{\extracolsep{\stretch{1}}}*{3}{l}@{}}
\toprule
  $\T$ & One-dimensional torus  & \\
  $V$ & Speed of the Kac particles \\
  $\lambda$ & Velocity switch rate of the Kac particles \\
  $\alpha \coloneqq V^2 / (2 \lambda)$ & Diffusivity in the FC system \\
  $\Omega_V$ & $\Omega_V \coloneqq \T \times \{-V, V\}$ & \\
  $\Omega_1$ & $\Omega_1 \coloneqq \T \times \{-1, 1\}$ & \\
  $\iota$ & Velocity-reversal map $\iota(x,v)=(x,-v)$ & Sec.~\ref{sec:Poisson-Kac} \\ 
  $\pi$ & Stationary measure for the Kac equation & Sec.~\ref{subsec:Kac-sol}\\
  $\mathcal L_{\mathcal X}$ & Lebesgue measure on the set $\mathcal X$ & \\
  $\operatorname{Unif}_{\!\mathcal Y}$ & Uniform measure on the set $\mathcal Y$ & \\
  $\calP(\mathcal X)$ & Space of probability measures on $\mathcal X$ & \\ 
  $\calM(\mathcal X)$ & Space of finite, signed, Borel measures on $\mathcal X$ &  \\
  $\|\cdot\|_{\TV}$ & Total-variation norm on measures & \eqref{def:TV-norm} \\
  $\dnabla$, $\ddiv$ & Gradient and divergence operators on $\Omega_V$ & \eqref{def:Omega-der}\\
  $\CE(0,T;\Omega_V)$ & Pairs $(\sigma,j)$ satisfying the continuity equation on $\Omega_V$ & Def.~\ref{def:CE}\\
  $\CE(0,T;\T)$ & Pairs $(\rho,\omega)$ satisfying the continuity equation on $\T$ & Lem.~\ref{lem:barOmega-ac}\\
  $\ME(0,T;\T)$ & Triples $(\sigma,\omega,K)$ satisfying the momentum equation on $\T$ & Def.~\ref{def:ME}\\
  $\scrI$ & Rate function for the Kac equation  & \eqref{eq:LDRF}\\
  $\scrJ$ & Rate function for the FC system & \eqref{IrhoomegaJ} \\
  $\scrH$, $\scrL$ & Hamiltonian and Lagrangian for the Kac equation & \eqref{Hamiltonian},~\eqref{eq:fin-Lag} \\
  $\scrI^V$ & Rate function used in the diffusive limit  & \eqref{eq:rescaled-RF}\\
  $\scrI^\lambda$ & Rate function used in the hyperbolic limit  & Sec.~\ref{subsec:hyperbolic}\\
  $\Ent(\cdot|\cdot)$ & Relative entropy  & \eqref{def:relEnt}\\
  $\FI(\cdot|\pi)$ & Fisher information with respect to $\pi$ & \eqref{def:FI}\\
  $d_{\BL}(\cdot,\cdot)$ & Bounded Lipschitz metric on probability measures & \eqref{eq:BL-metric} \\
  \bottomrule                            
\end{longtable}
\end{center}

Throughout, we use common measure-theoretic notation and terminology. For a measure $\sigma\in \calM([0,T]\times\mathcal{X})$, for instance, we often
write $\sigma_t\in\calM(\mathcal X)$  for the time slice at time $t$; we also often use both the notation $\rho(x)\,dx$ and $\rho(dx)$ when $\rho$ has Lebesgue density. We equip $\calM(\mathcal X)$ and $\calP(\mathcal X)$ with the narrow topology, in which the convergence is characterized by duality with continuous and bounded functions on $\mathcal X$. We equip $C(B;\calP(\mathcal X))$ and $C(B;\calM(\mathcal X))$ with the uniform topology in $B\subseteq \mathbb R$ and the narrow topology in $\mathcal X$.


\section{Solution concepts, continuity equation, and variational formulation}\label{sec:variational}

The Kac equation~\eqref{marginals} and the FC system~\eqref{damped_hyperbolic} are the two main evolution equations studied in this article. In what follows, we introduce two solution concepts---the one of a weak solution and of a variational solution---for the Kac equation, where the latter makes use of a variational structure. Theorem~\ref{prop:exist} discusses the equivalence of these two notions. 
Both these solution concepts carry over to the FC system using the bijection~\eqref{eq:intro-biject} as is clarified in Theorem~\ref{thm:Kac-wellposed}.

\subsection{Solution concepts for the Kac equation}\label{subsec:Kac-sol}

 Recall the Kac equation
 \begin{align}\label{eq:strong-form}
   \left\{\quad \begin{aligned} \partial_t\sigma + v\,\partial_x \sigma &= \lambda \, (\iota_\sharp\sigma - \sigma) \, , \\
        \sigma|_{t=0} &= \sigma_0\in\calP(\Omega_V) \, ,
    \end{aligned}\right.
 \end{align}
where $(\iota_\sharp\sigma)(\cdot,v)=\sigma(\cdot,-v)$.
Note that this evolution admits the uniform distribution $\stat \coloneqq \calL_{\T} \otimes \operatorname{Unif}_{\!\{-V, V\}}$ as the unique invariant measure.

\begin{defi}[Weak solution]\label{def:weak-sol-Kac}
The curve $\sigma\in C([0,T];\mathcal P(\Omega_V))$ in the space of probability measures is a \emph{weak solution} to the Kac equation~\eqref{eq:strong-form} if
\begin{enumerate}
    \item $\sigma|_{t=0}=\sigma_0$, 
    \item for any $\varphi\in C^{1,0}(\Omega_V)$ and $0 \leq s \leq t \leq T$,  
 \begin{align}\label{eq:weak-formulation}
     \int_{\Omega_V} \varphi(x,v)\,\sigma_t(dxdv) - \int_{\Omega_V} \varphi(x,v)\,\sigma_s(dxdv) = \int_s^t \int_{\Omega_V} (Q\varphi)(x,v)\,\sigma_r(dxdv)\,dr \, ,
 \end{align}
 where the dependence on time is indicated in the subscript and the generator $Q$ is defined as
 \begin{equation}\label{eq:gen}
 (Q\varphi)(x,v)\coloneqq v\,\partial_x \varphi(x,v) + \lambda \, \bigl( (\varphi\circ\iota)(x,v)-\varphi(x,v) \bigr) \, .
 \end{equation}
\end{enumerate}
\end{defi} 
 
\medskip
 
The existence and uniqueness of weak solutions to the Kac equation will be discussed at the end of this section in Theorem~\ref{prop:exist}. In what follows, we will often make use of the following characterisation of the  total-variation (TV) norm. For $\sigma\in\mathcal M(\mathcal X;\R^d)\coloneqq \{\sigma=(\sigma_1,\ldots,\sigma_d): \sigma_i\in\calM(\mathcal X) \text{ for } 1\leq i\leq d \}$,
\begin{equation}\label{def:TV-norm}
\lVert\sigma\rVert_{\TV} \coloneqq \sup
\Bigl\{ \int_{\mathcal X}\varphi\cdot\,d\sigma : \varphi \in C(\mathcal X;\R^d),\; \lvert \varphi_i \rvert \leq 1 \;\;\text{for all $i=1,\ldots,d$}\Bigr\}\,.
\end{equation}

We now introduce the notion of a continuity equation which connects a flux~$j=(j^1, j^2)$ to a probability measure~$\sigma$. Such concepts are standard in nonequilibrium thermodynamics \cite[Chapter~II]{dGM84} (it is a special case of a so-called ``balance equation'' without a source term) and variational literature \cite[Section~8.1]{AmbrosioGigliSavare08}, \cite[Def.~4.1]{PRST22}.

\begin{defi}[\textbf{C}ontinuity \textbf{E}quation]\label{def:CE} The pair $(\sigma,j)\in\CE(0,T;\Omega_V)$  if
\begin{enumerate}
    \item $\sigma\in C([0,T];\calP(\Omega_V))$,
    \item $(j_t)_{t\in (0,T)}\subset \calM( \Omega_V;\R^2)$ is a measurable family satisfying
    \[
        \int_0^T \|j_t\|_{\TV}\,dt <\infty\,,
    \]
    \item\label{item:CE-weak} for any $\varphi\in C^{1,0}(\Omega_V)$ and $0\leq s\leq t\leq T$, 
 \begin{equation}\label{eq:CE-IntegralForm}
  \langle \varphi, \sigma_t\rangle - \langle \varphi, \sigma_s \rangle = \int_s^t \langle \dnabla \varphi, j_r\rangle\, dr \, ,
\end{equation}
where $\langle a,b \rangle=\int_{\Omega_V}a \cdot db$ and the divergence and gradient operators are defined as
 \begin{equation}\label{def:Omega-der}
 \dnabla \varphi \coloneqq  \bigl(\partial_x \varphi, \dpartial_v\varphi\bigr) \, , \quad -\ddiv j \coloneqq -\partial_x j^1 + \iota_\sharp j^2 - j^2
\, , \quad \dpartial_v\varphi\coloneqq\varphi \circ \iota - \varphi \, .
 \end{equation}
\end{enumerate}
\end{defi}

\begin{rem}\label{rem:sol-CE} 
A weak solution (see Definition~\ref{def:weak-sol-Kac}) to the the Kac equation~\eqref{eq:strong-form} can be written in the form~\eqref{eq:CE-IntegralForm} with $j\coloneqq(v\,\sigma,\lambda\,\sigma)$, since  
 \begin{align*}
     \int_{\Omega_V} (Q\varphi)(x,v)\,\sigma_t(dxdv) &= \int_{\Omega_V} \bigl(v\,\partial_x \varphi(x,v) + \lambda \, \dpartial_v\varphi(x,v)\bigr)\,\sigma_t(dxdv) \\
     &= \int_{\Omega_V} \dnabla\varphi(x,v) \cdot (v, \lambda) \,\sigma_t(dxdv) = \int_\Omega \dnabla\varphi\cdot d j_t \, .
 \end{align*}
 Therefore, a weak solution to the Kac equation with initial datum~$\sigma_0$ also satisfies the continuity equation with  $\bigl(\sigma, (v\,\sigma, \lambda\,\sigma)\bigr) \in \CE(0, T; \Omega_V)$.
\end{rem}

The continuity equation above is defined in terms of time-independent test functions. However, in the proofs of asymptotic limits in Section~\ref{sec:prop_lim} (Lemma~\ref{lem:rho-bv} in particular) we will need to use time-dependent test functions in the continuity equation because of the lack of control on the temporal regularity of the fluxes.
\begin{lem}\label{lem:time-space-weak} Fix $V\geq 1$ and $(\sigma,j)\in \CE(0,T;\Omega_V)$. For any $\chi\in C^1_c((0,T))$ and $\varphi\in C^{1,0}(\Omega_V)$:
\begin{equation*}
    \int_0^T\int_{\Omega_V}\dot\chi(t)\,\varphi(x,v)\,\sigma_t(dxdv)\,dt = -\int_0^T\int_{\Omega_V}\chi(t)\, \dnabla\varphi(x,v)\cdot j_t(dxdv)\,dt \, .
\end{equation*}
\end{lem}
\begin{proof}
For any $\chi\in C^1_c((0,T))$, $\varphi\in C^{1,0}(\Omega_V)$, and for sufficiently small $h>0$, we find
\begin{align*}
    \int_0^T&\int_{\Omega_V} \mathbb 1_{[h,T]}(t) \, \frac{\chi(t)-\chi(t-h)}{h} \, \varphi(x,v) \, \sigma_t(dxdv) \, dt = \\ &=\frac1h\int_h^T\bigl(\chi(t)-\chi(t-h)\bigr) \, \langle \varphi,\sigma_t \rangle \, dt \\
    &= \frac1h\int_h^T\chi(t) \, \langle \varphi,\sigma_t \rangle \, dt - \frac1h\int_0^{T-h}\chi(t) \, \langle \varphi,\sigma_{t+h} \rangle \, dt\\
    &= - \frac1h\int_{h}^{T-h}\chi(t) \, \bigl(\langle \varphi,\sigma_{t+h} \rangle-\langle \varphi,\sigma_{t} \rangle \bigr) \, dt - \frac1h \int_0^h \chi(t) \, \langle \varphi,\sigma_{t+h} \rangle \, dt + \frac1h \int_{T-h}^{T} \chi(t) \, \langle \varphi,\sigma_{t+h} \rangle \, dt \\
    &=- \int_{0}^{T}\mathbb 1_{[h,T-h]}(t) \, \chi(t) \, \frac1h\int_t^{t+h}\langle \dnabla\varphi,j_s \rangle \, ds\, dt \, ,
\end{align*}
where the final equality follows since $(\sigma,j)\in\CE(0,T;\Omega_V)$ and the last two terms vanish since $h$ may be chosen so that $[h, T-h]$ fully contains the support of~$\chi$.


Note that by the definition of the continuity equation, $t\mapsto \|j_t\|_{\TV} \in L^1((0,T))$ and therefore $s \mapsto \langle \dnabla\varphi, j_s\rangle$ also belongs to $L^1((0, T))$. In the following we will prove that 
\begin{equation*}
     \int_{0}^{T}\mathbb 1_{[h,T-h]}(t) \, \chi(t)  \frac1h\int_t^{t+h}\langle \dnabla\varphi,j_s \rangle \, ds dt  \xrightarrow{h\rightarrow 0} \int_{0}^{T}\mathbb \chi(t) \, \langle \dnabla\varphi, j_t \rangle \, dt \, .
\end{equation*}
To prove this, we only need to show that for any $f\in L^1_{\mathrm{loc}}((0,T))$ we have
\[
    g^h(t):=\frac{1}{h}\int_t^{t+h} f(s)\,ds \longrightarrow f(t)\qquad\text{in }L^1_{loc}((0, T)) \, ,
\]
i.e., $\|g^h-f\|_{L^1(K)}\to 0$ as $h\to 0$ for every compact set $K\subset (0,T)$. By Lusin's theorem, we find a sequence $(f_j)_j\subset \calC_c((0,T))$ satisfying
\[
    \|f_j-f\|_{L^1((0, T))}\to 0 \, .
\]
Furthermore, for each $j\in\mathbb{N}$ and $h\ll 1$ sufficiently small,
\begin{align*}
    A_j^h \coloneqq \int_0^T \biggl\lvert\frac{1}{h}\int_t^{t+h} f_j(s)\,ds - f_j(t)\biggr\rvert \, dt
    \le \int_0^{1}  \int_0^T \lvert f_j(t+hs)-f_j(t)\rvert \,dt\,ds \;\;\xrightarrow{h\to 0}\;\; 0 \, ,
\end{align*}
due to the uniform continuity of $f_j$. An application of the triangle inequality yields
\begin{align*}
    \|g^h - f\|_{L^1(K)} \le \biggl\| g^h - \frac{1}{h}\int_{\cdot}^{\cdot+h} f_j(s)\,ds\biggr\|_{L^1(K)} + A_j^h + \|f_j-f\|_{L^1((0, T))} \, .
\end{align*}
For $h\ll 1$, the first term may be bounded from above by
\[
 \int_0^{1} \int_h^{T-h}|f(hs+t) - f_j(hs+t)|\,dt\,ds = \int_0^1 \int_{h(s+1)}^{T+ h(s-1)} |f(t) - f_j(t)|\,dt\,ds \le \|f-f_j\|_{L^1((0, T))} \,.
\]
Consequently, we can pass first to the limit $h\to 0$ and then $j\to \infty$ to deduce the asserted convergence.

Since $\chi\in C^1_c((0,T))$, using the dominated convergence theorem (and the mean value theorem to provide an upper bound), we find
\begin{equation*}
    \lim_{h\rightarrow 0}\int_0^T\int_{\Omega_V} \mathbb 1_{[h,T]}(t) \, \frac{\chi(t)-\chi(t-h)}{h} \, \varphi(x,v) \, \sigma_t(dxdv)dt=\int_0^T\int_{\Omega_V}  \dot\chi(t) \, \varphi(x,v) \, \sigma_t(dxdv) \, dt
\end{equation*}
and thus arrive at the required result.
\end{proof}

\subsection{Variational structure for the Kac equation}\label{subsec-Kac-var}
The goal of this section is to introduce a variational formulation for the Kac equation \eqref{eq:strong-form} that will (i) induce a variational structure on the FC system (Section~\ref{sec:Var}) and (ii) be used to perform the asymptotic limits in Section~\ref{sec:Limits}.

We define the functional $\scrI \colon C([0,T];\calP(\Omega_V))\times \calM([0,T];\calM(\Omega_V;\R^2))\rightarrow[0,+\infty]$ by
\begin{equation}\label{eq:LDRF}
\scrI(\sigma,j)\coloneqq\begin{dcases*}
\displaystyle
\int_0^T \scrL(\sigma_t,j_t) \, dt \ \ &if $(\sigma,j)\in \CE(0,T;\Omega_V)$\,,\\
+\infty &otherwise.
\end{dcases*}
\end{equation}
Since this functional is inspired by the large-deviation rate functional corresponding to the Kac process (see \appendixname~\ref{sec:LD}), hereafter we will refer to~\eqref{eq:LDRF} as the rate functional. 

The Lagrangian $\scrL \colon \calP(\Omega_V) \times \calM(\Omega_V;\R^2) \to [0, +\infty]$,
\begin{equation}\label{def:Lag-CE}
    \scrL(\sigma,j) = \sup_{\varphi\in C(\Omega_V; \R^2)} \! \bigl( \langle\varphi,j\rangle - \scrH(\sigma,\varphi)\bigr) \, ,
\end{equation}
is the Legendre dual of the Hamiltonian
\begin{equation}\label{Hamiltonian}
 \scrH(\sigma,\varphi) = \int_{\Omega_V} \bigl[ v\, \varphi_1(x,v) + \lambda \, \bigl( e^{\varphi_2(x,v)}-1\bigr) \bigr] \, \sigma(dxdv) \, .
\end{equation}
Since
\begin{equation*}
\langle \varphi,j\rangle - \scrH(\sigma,\varphi) = \int_{\Omega_V} \varphi_1\,\bigl( j^1 - v\,\sigma\bigr) + \int_{\Omega_V} \bigl[ \varphi_2\, j^2 - \lambda \, \bigl(e^{\varphi_2}-1\bigr) \, \sigma \bigr] \, ,
\end{equation*}
we deduce that
\begin{equation}\label{eq:fin-Lag}
\begin{aligned}
    \scrL(\sigma,j) &= \begin{cases}
     \displaystyle\sup\limits_{\varphi\in C(\Omega_V)} \int_{\Omega_V} \bigl[ \varphi\, dj^2 - \lambda\bigl(e^\varphi-1\bigr)  \, d\sigma \bigr] &\text{if $j^1=v\,\sigma$}, \\
     +\infty & \text{otherwise,} 
    \end{cases} \\
    &= \begin{cases}
        \Ent(j^2 \mid \lambda \, \sigma) &\text{if $j^1=v\,\sigma$}, \\
        +\infty & \text{otherwise},
    \end{cases}
\end{aligned}
\end{equation}
where $\Ent(\cdot|\cdot)$ is the relative entropy on $\calM(\Omega_V)\times\calM(\Omega_V)$, defined as
\begin{equation}\label{def:relEnt}
\Ent(\mu\mid\nu)\coloneqq
\begin{dcases} 
\int_{\Omega_V}(f\log f -f +1 )\,d\nu \quad & \text{if } \mu\ll \nu \text{ with } f\coloneqq \frac{d\mu}{d\nu} \, ,
\\
+\infty & \text{otherwise}.
\end{dcases}
\end{equation}
Note that $\scrH$ is convex in the second argument and, therefore, $\scrL(\sigma,\cdot)$ and $\scrH(\sigma,\cdot)$ are convex bi-duals. Furthermore, $\scrI\geq 0$, which is seen by choosing $\varphi=0$ in~\eqref{eq:fin-Lag}. 

\begin{rem}
 In the Hamiltonian~\eqref{Hamiltonian}, the cotangent vectors~$\varphi$ are functions on the state space~$\Omega_V$ instead of functions on the `space of edges' $T\Omega_V \cong \Omega_V \times \Omega_V$, as would be expected in the general case of jump processes \cite{PRST22}. This discrepancy is due to the identification of jump kernels on $\Omega_V$ with measures on $\Omega_V$ that we make at the end of \appendixname~\ref{sec:LD}.
\end{rem}
We now introduce the notion of a variational solution for the Kac equation as the zero level set of the rate functional~\eqref{eq:LDRF}.
\begin{defi}[Variational solution]
The curve $\sigma\in C([0,T];\calP(\Omega_V))$ is a \emph{variational solution} to the Kac equation~\eqref{eq:strong-form} if there exists a measurable family $(j_t)_{t\in(0,T)}\subset\calM(\Omega_V;\R^2)$ such that the pair $(\sigma,j)\in\CE(0,T;\Omega_V)$ and 
\[
    \Ent\bigl(\sigma|_{t=0}\mid\sigma_0\bigr) + \scrI(\sigma,j) =0 \, .
\]
\end{defi}
The following result discusses the existence and uniqueness of solution to the Kac equation~\eqref{eq:strong-form} and the equivalence of the two solution concepts introduced above.  
\begin{theorem}\label{prop:exist}
Given $\sigma_0\in\calP(\Omega_V)$, there exists a unique weak solution $\sigma\in C([0,T];\calP(\Omega_V))$ to the Kac equation~\eqref{eq:strong-form}. Moreover, $\sigma$ is a weak solution of~\eqref{eq:strong-form} if and only if it is a variational solution.
\end{theorem}
\begin{proof}
Since the solution to the Kac equation is the law of a Markov process with a generator that satisfies the maximum principle, classical results~\cite[Chapter~4]{EthierKurtz09} imply the existence of a unique martingale solution, which in turn implies the existence of a unique weak solution. 

We now discuss the equivalence of the two solution concepts. Assume that $\sigma\in C([0,T];\calP(\Omega_V))$ is a weak solution to~\eqref{eq:strong-form}. Then, using Remark~\ref{rem:sol-CE} with the choice $j_t=(v\,\sigma_t,\lambda\,\sigma_t)$ for any $t\in [0,T]$, the pair $(\sigma,j)\in \CE(0,T;\Omega_V)$. This choice yields $\scrI(\sigma,j)=0$ and, since $\sigma_t\rightarrow\sigma_0$ narrowly as $t\rightarrow 0$, we also have $\Ent(\sigma|_{t=0}|\sigma_0)=0$.  Therefore, $\sigma$ is a variational solution. 
We now assume that $\sigma\in C([0,T];\calP(\Omega_V))$ is a variational solution to~\eqref{eq:strong-form}. Since $\scrL\geq 0$, it follows that $j_t=(v\,\sigma_t,\lambda\,\sigma_t)$ for almost every $t\in [0,T]$. 
Since $(\sigma,j)\in \CE(0,T;\Omega_V)$, using Definition~\ref{def:CE}(\ref{item:CE-weak}), we conclude that $\sigma$ is a weak solution of~\eqref{eq:strong-form}.
\end{proof}

\subsection{Variational structure for the FC system}\label{sec:Var}

In this section, we discuss the implications of a finite rate function $\scrI$ for the FC system~\eqref{damped_hyperbolic}. To do so, we make a change of variables from the probability measure $\sigma$ to $(\rho,\omega)$ and the corresponding fluxes. 

We begin by defining the bijection $\Pi_V\colon\calM(\Omega_V)\to \calM(\T)\times \calM(\T)$ as
\[
    \bigl(\Pi_Vj\bigr)(dx) \coloneqq \Bigl( \sum_{v\in\{-V,V\}}\!j(dx,v)\, , \sum_{v\in\{-V,V\}}\! v\,j(dx,v) \Bigr) \, ,
\]
with inverse
\[
    \bigl(\Pi_V^{-1} J\bigr)(dx,v) = \frac{1}{2}\left(J_1(dx) + \frac{1}{v}J_2(dx)\right),\qquad J=(J_1,J_2)\,.
\]
The \emph{density}~$\rho$ and the \emph{flux}~$\omega$, defined as (recall the motivating discussion in Section~\ref{sec:Poisson-Kac})
\begin{equation*}
    \rho(\di x) \coloneqq \sum\limits_{v \in \{-V, V\}} \! \sigma(\di x, v) \, ,\qquad
    \omega(\di x) \coloneqq \sum\limits_{v \in \{-V, V\}}\! v \, \sigma(\di x, v) \,,
\end{equation*}
are then given by $(\rho,\omega)=\Pi_V\sigma$.

Now, let $(\sigma,j)\in \CE(0,T;\Omega_V)$. Formally multiplying the continuity equation for $(\sigma,j)$ by $(1,v)$ and summing over $v\in\{-V,V\}$, we obtain the following linear system
\begin{equation}\label{CE_proj}
\begin{aligned}
    \partial_t \rho + \partial_x J_1^1 &= 0 \, , \\
    \partial_t \omega + \partial_x J_2^1 &= - 2 J^2_2 \, ,
\end{aligned}\qquad \text{with}\quad (\rho,\omega) = \Pi_V\sigma,\quad J^i = \Pi_Vj^i,\; i=1,2\,.
\end{equation}
Suppose that $\scrI(\sigma,j)<\infty$ with $\sigma= \Pi_V^{-1}(\rho,\omega)$ and $j^1=\Pi_V^{-1} J^1$. Then, the condition
\[
    (\Pi_V^{-1}J^1)(dx,v) = j^1(dx,v) = v\,\sigma(dx,v) =  v\,\bigl(\Pi_V^{-1}(\rho,\omega)\bigr)(dx,v)
\]
necessarily implies $J_1^1 = \omega$ and $J_2^1 = V^2\rho$. 

We then arrive at the functional 
$\scrJ\colon C([0,T];\calP(\T))\times C((0,T);\calM(\T)) \times \calM((0,T);\calM(\T;\R^2))\to [0,+\infty]$ given by
\begin{equation}\label{IrhoomegaJ}
    \scrJ(\rho,\omega,J) \coloneqq \begin{dcases*}
        \int_0^T \Ent\bigl(\Pi_V^{-1}J_t\mid \lambda\,\Pi_V^{-1}(\rho_t,\omega_t)\bigr) \, dt, \ \ &if $(\rho,\omega,J_2)\in \text{ME}(0,T;\T)$,\\
        +\infty, &otherwise,
\end{dcases*}
\end{equation}
where
$\text{ME}(0,T;\T)$ is the class of solutions to the linear ``momentum'' system
\begin{equation*}
    \partial_t \rho + \partial_x \omega = 0 \, , \qquad 
    \partial_t \omega + V^2\partial_x \rho = -2 J_2 \, ,
\end{equation*}
in the following sense.
\begin{defi}[\textbf{M}omentum \textbf{E}quation]\label{def:ME}
The triple $(\rho,\omega,K)\in \ME(0,T;\T)$ if
\begin{enumerate}
    \item $(\rho,\omega)\in C([0,T];\calP(\T))\times C((0,T);\calM(\T))$ 
    \item $(K_t)_{t\in (0,T)}\subset \calM( \T)$ is a measurable family satisfying
        \[
            \int_0^T \|K_t\|_{\TV}\,dt <\infty,
        \]
    \item\label{item:ME-weak} for any $\varphi,\psi\in C^{1}(\T)$ and $0 \leq s \leq t \leq T$,  
\begin{subequations}\label{eq:weak-FC}
\begin{align}
     \int_{\T} \varphi(x)\,\rho_t(dx) - \int_{\T} \varphi(x)\,\rho_s(dx) &= \int_s^t \int_{\T} \partial_x\varphi(x) \, \omega_r(dx) \, dr \, ,\\
     \int_{\T} \psi(x)\,\omega_t(dx) - \int_{\T} \psi(x)\,\omega_s(dx) &= V^2\int_s^t \int_{\T} \partial_x\psi(x) \, \rho_r(dx) \, dr - 2\int_s^t \int_{\T} \psi(x) \, K_r(dx) \, dr\,. 
 \end{align}
\end{subequations}
\end{enumerate}
\end{defi}

We now define the notion of a variational solution for the FC system.
\begin{defi}[Variational solution]
 The pair $(\rho, \omega)$ is a variational solution to the FC system~\eqref{damped_hyperbolic} if there exists a measurable family $(J_t)_{t\in(0,T)} \subset \calM(\T; \R^2)$ such that $(\rho, \omega, J_2) \in \ME(0, T; \T)$ and
 \begin{equation*}
  \Ent\bigl(\Pi_V^{-1}(\rho, \omega)\rvert_{t=0} \mid \Pi_V^{-1}(\rho_0, \omega_0)\bigr) + \scrJ(\rho, \omega, J) = 0 \, .
 \end{equation*}
\end{defi}

Notice that if $\scrI(\sigma,j)=0$, then also $\scrJ(\rho,\omega,\Pi_V j) =0$ with $(\rho,\omega)=\Pi_V\sigma$. Hence, a variational solution $(\sigma,j)$ of the Kac equation \eqref{eq:strong-form} gives a variational solution to the FC system \eqref{damped_hyperbolic}. Moreover, observe that $\scrJ(\rho,\omega,J) = 0$ implies $J_2=\lambda\,\omega$, and we recover a weak solution of the FC system~\eqref{damped_hyperbolic}, which we introduce next together with the well-posedness.

\begin{defi}[Weak solution]
The pair $(\rho,\omega)$ is a weak solution to the FC system~\eqref{damped_hyperbolic} with initial datum $(\rho_0,\omega_0)\in\calP(\T)\times\calM(\T)$ if $(\rho,\omega,\lambda\,\omega)\in\ME(0,T;\T)$ with  $(\rho,\omega)\rvert_{t=0} = (\rho_0,\omega_0)$.
\end{defi}

\begin{theorem}\label{thm:FC-wellposed}
Consider the initial datum $(\rho_0,\omega_0)\in\calP(\T)\times\calM(\T)$ satisfying the bounded-speed condition
\begin{equation}\label{eq:cone-condition}
    -V\rho_0\leq \omega_0 \leq V\rho_0\,.
\end{equation}
Then, there exists a unique weak solution to the FC system~\eqref{damped_hyperbolic}.
\end{theorem}
\begin{proof}
By the assumption on the initial datum, $\sigma_0\coloneqq \Pi_V^{-1}(\rho_0,\omega_0) \in \calP(\Omega_V)$. By Theorem~\ref{prop:exist}, there exists a unique weak solution $\sigma\in C([0,T];\calP(\Omega_V))$ to the Kac equation~\eqref{eq:strong-form}. Hence, the pair $\Pi_V\sigma\eqqcolon(\rho,\omega)\in C([0,T];\calP(\T))\times C([0,T];\calM(\T))$ satisfies~\eqref{eq:weak-FC} with $K=\lambda\, \omega$, and therefore is a weak solution to the FC system~\eqref{damped_hyperbolic} with initial datum $(\rho_0,\omega_0)$.

Let $(\rho^1,\omega^1)$ and $(\rho^2,\omega^2)$ be two weak solutions to the FC system with initial datum  $(\rho_0,\omega_0)$. It follows that  $\sigma^i\coloneqq\Pi_V^{-1}(\rho^i,\omega^i)$, $i=1,2$, are both weak solutions to the Kac equation~\eqref{eq:strong-form}. From the uniqueness of the weak solution to the Kac equation (Theorem~\ref{prop:exist}), we find 
\begin{equation*}
    \Pi_V^{-1}(\rho^1,\omega^1) = \Pi_V^{-1}(\rho^2,\omega^2)  \quad  \text{in\; $\calP(\Omega_V)$\; for all $t\in [0,T]$}\,,
\end{equation*}
thus implying that $\rho^1=\rho^2$ and $\omega^1=\omega^2$.
\end{proof}

The following result makes the equivalence of the Kac equation and the FC system precise and follows on the lines of the proof above. 
\begin{theorem}\label{thm:Kac-wellposed}
Let $\sigma$ be the weak solution to the Kac equation \eqref{eq:strong-form} with initial datum $\sigma_0\in\calP(\Omega_V)$. Then, $(\rho,\omega)\coloneqq\Pi_V\sigma$ is the weak solution to the FC system~\eqref{damped_hyperbolic} with initial data $(\rho_0,\omega_0)=\Pi_V\sigma_0$. 

Conversely, if the pair $(\rho,\omega)$ is the weak solution to the FC system~\eqref{damped_hyperbolic} with initial datum $(\rho_0,\omega_0)\in\calP(\T)\times\calM(\T)$ such that 
\begin{equation}\label{ass:bound-flux}
    -V\rho_0\leq \omega_0 \leq V\rho_0\,,
\end{equation}
then $\sigma \coloneqq \Pi_V^{-1}(\rho, \omega)$ is the weak solution to the Kac equation~\eqref{eq:strong-form} with initial datum $\sigma_0=\Pi_V^{-1}(\rho_0,\omega_0)$.  
\end{theorem}

The following remarks discuss the bounded-flux assumption~\eqref{ass:bound-flux} on the initial flux and the literature related to the FC system. 

\begin{rem}
The condition $-V \rho_0 \le \omega_0 \le V \rho_0$ ensures that $\sigma_0 \coloneqq \Pi_V^{-1}(\rho_0, \omega_0)$ a probability measure. It propagates to all times and implies that, for any given measurable set, (i) the system cannot transport more mass than the mass contained in that set, and (ii) the maximum speed at which the mass is transported does not exceed $V$ since $|d\omega_0/d\rho_0|\le V$. This condition is not a distinctive feature of the FC system, but originates from its connection to the Kac equation, i.e., the solutions to the Kac equation and the FC system can be connected only under this bounded-flux assumption at initial time. General FC systems, however, may have solutions~$\rho$ that are not probability measures but rather signed measures or Sobolev functions. The latter is typical of the standard hyperbolic literature which works with initial data in Sobolev spaces~\cite{BenzoniSerre06}.
\end{rem}

\begin{rem}
 The FC system is related to the partially damped isothermal compressible Euler equations, where an additional convective term is present in \eqref{flux}. Global bounded solutions exist for initial data satisfying condition \eqref{eq:cone-condition} (cf.\ \cite[Section~3]{Zhao2010}). 
 In \cite[Section~5.4]{PKG18} similar models, but for hyperbolic heat transport, are constructed in the form of GENERIC and differ from the FC system. The FC system, instead, possesses only a weaker version known as pre-GENERIC \cite{KZMP20}. Since showing this fact requires the introduction of additional notation, we postpone the pre-GENERIC structure of the FC system to \appendixname~\ref{sec:preGEN}, which may be of independent interest.
\end{rem}

\section{FIR inequality}\label{sec:FIR}
In the last section we introduced a variational structure for the Kac equation by which we defined a variational solution as its zero level set. As we shall see in the rest of this article, this variational structure also allows us to study \emph{approximate} solutions, which correspond to the non-zero level sets of the rate functional. 
The regularity properties of such sub-level sets are made explicit by an \emph{a priori} estimate that, for the Kac equation, connects the relative entropy and the Fisher information (defined below) to the rate functional.
This estimate will play a crucial role in studying asymptotic 
limits in Section~\ref{sec:Limits}. Specifically, this inequality provides control on the Fisher information (which encodes regularity properties of the flux) in terms of the values of the rate functional. 

To present this estimate, we first define the \emph{Fisher information } $\FI(\cdot|\pi) \colon \calP(\Omega_V)\rightarrow[0,+\infty]$ as
 \begin{equation}\label{def:FI}
 \FI(\eta\mid\pi)\coloneqq \begin{dcases*}\displaystyle
 \frac12 \int_{\Omega_V} \biggl( \sqrt{\frac{d\eta}{d\pi}\circ \iota} -\sqrt{ \frac{d\eta}{d\pi}}\biggr)^{\!2} \, d\stat &if $\eta\ll \pi$, \\
 +\infty &\text{otherwise,}
 \end{dcases*}
 \end{equation}
 where $\stat$ is the invariant measure for the Kac equation~\eqref{eq:strong-form}. The Fisher information has several useful properties, such as non-negativity, convexity, and lower semicontinuity, which are summarized in Proposition~\ref{lem:FI-prop} below. It is closely related to entropy dissipation  and is a natural object that appears in the variational approaches of~\cite{DLPS17,HPST20,PRST22}. We now state the FIR inequality.
 
\begin{theorem}\label{thm:FIR}
Consider a pair $(\sigma,j)\in \CE(0,T;\Omega_V)$ satisfying
\begin{equation}\label{well-preparedness}
\Ent(\sigma_0\mid\stat) + \scrI(\sigma,j) < \infty \, ,
\end{equation}
with $\sigma|_{t=0}=\sigma_0$. Then, for any $t\in [0,T]$, we have
\begin{equation}\label{eq:FIR-pi}
\Ent(\sigma_t\mid\stat) + \lambda \int_0^t \FI(\sigma_r\mid\stat)\,dr \leq \Ent(\sigma_0\mid\stat) + \scrI(\sigma,j) \, .
\end{equation}
\end{theorem}
An obvious consequence of Theorem~\ref{thm:FIR} is that the relative entropy with respect to the stationary measure is a Lyapunov function for the Kac equation, as we may verify by choosing $\scrI(\sigma,j)=0$ and by the positivity of the Fisher information.
We comment on the assumptions of Theorem~\ref{thm:FIR} in the following remark.

\begin{rem}\label{rem:well-preparedness}
The initial datum being \emph{well-prepared} via $\Ent(\sigma_0\mid\pi)<\infty$ implies that, in the $x$-variable, the initial data $\sigma_0(\cdot,v) \ll \mathcal L_{\T}$ for all $v$. Since the Kac equation is well-posed for a considerably larger class of initial data (cf.~Theorem~\ref{thm:Kac-wellposed} and Figure~\ref{fig:delta} with a Dirac initial datum), we expect that this assumption can be relaxed to allow for such singular initial data---we give formal arguments for this observation in Remark~\ref{rem:genFIR-formal}. Making these formal arguments rigorous would require significant technical machinery which we wish to avoid in this article both to simplify the presentation and since it would not considerably improve the underlying understanding of the system. 

The assumption that the rate functional $\scrI(\sigma,j)$ is bounded arises naturally in the context of the large-deviation principle (cf.~\appendixname~\ref{sec:LD}), wherein it implies that the pair $(\sigma,j)$ solves the Kac equation approximately. In other words, such a pair is a fluctuation around a variational solution, which is the zero level set of the rate functional. Intuitively, Theorem~\ref{thm:FIR} states that the  connection between entropy and Fisher information not only applies to solutions, where the Fisher information quantifies the rate of decay of entropy, but also to fluctuations (in the large-deviation sense) around solutions. In Section~\ref{sec:Limits}, the FIR inequality will play a central role in studying asymptotic limits, and a consequence of this bounded-rate-functional assumption is that we study asymptotic convergence of both solutions and fluctuations.  
\end{rem}
 
We now illustrate the intuitive ideas behind the proof of Theorem~\ref{thm:FIR}. The heuristic motivation makes use of an appropriate choice for the test functions $\varphi$ in the dual formulation for the Lagrangian~\eqref{eq:fin-Lag}. Assuming that $\sigma$ has a smooth density $\sigma_t$ in time, we formally calculate
\begin{align*}
    \frac12\frac{d}{dt} \int_{\Omega_V} \sigma_t\log\frac{\sigma_t}{\pi} &= \frac12\int_{\Omega_V} \partial_t\sigma_t \, \log \frac{\sigma_t}{\pi} + \frac12\int_{\Omega_V} \partial_t \sigma_t \\
    & = \frac12\int_{\Omega_V} \partial_x\Bigl(\log\frac{\sigma_t}{\pi}\Bigr) \, j^1_t + \frac12\int_{\Omega_V} \dpartial_v\Bigl(\log\frac{\sigma_t}{\pi}\Bigr) \, j^2_t + 0 \\
    & = \frac12\int_{\Omega_V} \partial_x\Bigl(\frac{\sigma_t}{\pi}\Bigr) \, v \, \pi + \frac12\int_{\Omega_V} \dpartial_v\Bigl(\log\frac{\sigma_t}{\pi}\Bigr) \, j^2_t \\
    & = 0 + \frac12\int_{\Omega_V} \dpartial_v\Bigl(\log\frac{\sigma_t}{\pi}\Bigr) \, j^2_t.
\end{align*}
The second equality follows since the pair $(\sigma,j)$ satisfies the continuity equation, and the zero follows since $\sigma_t\in \calP(\Omega_V)$ for every $t$. The third equality follows since $\scrI(\sigma,j)<\infty$ implies that $j^1_t=v\,\sigma_t$, and the zero in the final equality follows by using integration by parts in the first integral. 
Using the variational form~\eqref{eq:fin-Lag} of the Lagrangian with the choice $\varphi=\frac12\dpartial_v\log(\sigma_t/\pi)$, the above calculation leads to 
\begin{equation}\label{eq:pre-genFIR}
    \frac12\frac{d}{dt} \int_{\Omega_V} \sigma_t\log\frac{\sigma_t}{\pi} \leq \scrL(\sigma_t,j_t) - \lambda \,\FI(\sigma_r\mid\pi) \, .
\end{equation}
Integrating in time over $[0,T]$, we arrive at the FIR inequality~\eqref{eq:FIR-pi}.

To make these calculations rigorous, we need to ensure that: (i) a chain rule holds for the map $t\mapsto \int_{\Omega_V}\sigma_t\log(\sigma_t/\pi)$, and (ii) this function is admissible in the dual formulation of the Lagrangian $\scrL$. Using Proposition~\ref{lem:FI-prop}, which collects some required properties of the Fisher information, in Lemma~\ref{lem:reg} we prove a general chain rule for appropriately regularised functions of measures. The proof of Theorem~\ref{thm:FIR} applies this lemma to a regularised version of $\sigma_t\log(\sigma_t/\pi)$ and then passes to the limit in the regularisation parameter to arrive at the FIR inequality.

\begin{prop}\label{lem:FI-prop} The Fisher information satisfies 
\begin{enumerate}[label=(\roman*)]
    \item $\FI(\cdot\mid\pi)\geq0$ on $\calP(\Omega_V)$ and $\FI(\eta\mid\pi)=0$ if and only if $\eta=\pi$; 
    \item $\FI(\cdot\mid\pi)$ is convex and weakly lower semicontinuous on $\calP(\Omega_V)$. 
\end{enumerate}
\end{prop}
We skip the proof since it follows by standard arguments that may be found, for instance, in~\cite{PRST22}. 
 
\begin{lem}\label{lem:reg}
Let $(\sigma,j)\in \CE(0,T;\Omega_V)$ with
\[
   \scrI(\sigma, j)= \int_0^T \scrL(\sigma_t,j_t)\,dt <\infty \, ,
\]
and, for any $t\in[0,T]$, define
\begin{equation*}
    \sigma_t^\eps(dxdv) \coloneqq \int_\T \Mol_\eps(x-y)\, \sigma_t(dydv)\,dx \, \quad \text{and} \quad j_t^\eps(dxdv) \coloneqq \int_\T \Mol_\eps(x-y) \,j_t(dydv)\,dx \, ,
\end{equation*}
where $\Mol_\eps$ is the heat kernel on $\T$, given by
\begin{equation}\label{eq:heatker-T}
    \Mol_\eps(x) \coloneqq \frac{1}{\sqrt{2\pi \eps}}\sum_{k\in\Z} e^{-\frac{\lvert x- k \rvert^2}{2\eps}}\qquad\text{for \; $\eps>0$\; and\; $x\in\T$\,.}
\end{equation}
Then, for every $\eps>0$, the pair $(\sigma^\eps,j^\eps)$ satisfies the following:
\begin{enumerate}
    \item $(\sigma^\eps,j^\eps)\in \CE(0,T;\Omega_V)$ with 
    \begin{align*}
        &\sigma_t^\eps \rightharpoonup \sigma_t \quad \text{narrowly in } \calP(\Omega_V) \text{ for all } t\in[0,T] \, , \\
        &  j_t^\eps \rightharpoonup j_t\quad \text{narrowly in } \calM(\Omega_V) \text{ for almost every } t\in[0,T] \, .
    \end{align*}
    
    \item $\scrL(\sigma_t^\eps,j_t^\eps) \le \scrL(\sigma_t,j_t)$\; for almost every $t\in[0,T]$.
    \item For any $\eps>0$, the curve $t\mapsto \sigma_t^\eps$ is absolutely continuous with respect to the total variation norm.
    \item Let $\phi\in C^2([0,\infty))$ 
    and $\calF \colon \calP(\Omega_V)\to[0,+\infty]$ be defined by
    \[
        \calF(\sigma) \coloneqq \begin{dcases*}
        \displaystyle\int_{\Omega_V} \phi\left(\frac{d\sigma}{d\pi}\right) d\pi & if $\sigma\ll \pi$, \\
        +\infty & otherwise.
        \end{dcases*}
    \]
    If $\sup_{t\in[0,T]} \calF(\sigma_t^\eps) <\infty$, then $(0,T)\ni t\mapsto \calF(\sigma_t^\eps)$ is absolutely continuous and the following chain rule holds:
    \[
        \frac{d}{dt}\calF(\sigma_t^\eps) = \int_{\Omega_V} \dpartial_v\phi'\Bigl(\frac{d\sigma_t^\eps}{d\pi}\Bigr) \, dj_t^{2,\eps}\qquad\text{for almost every } t\in(0,T) \, .
    \]
\end{enumerate}
\end{lem} 
\begin{proof}
    \emph{Ad }(1):  From the properties of the heat kernel $\Mol_\eps$, it is not difficult to see that the pair $(\sigma^\eps,j^\eps)$ satisfies the continuity equation. Moreover, for any test function $\varphi\in C(\Omega_V)$ and $t\in [0,T]$, we have 
    \begin{equation*}
        \langle \varphi, \sigma^\eps_t \rangle = \langle \varphi^\eps, \sigma_t\rangle \xrightarrow{\eps\rightarrow 0} \langle \varphi, \sigma_t\rangle\,,
    \end{equation*}
    with
    \begin{equation*}
        \varphi^\eps(x, v) \coloneqq \int_{\T} \Mol_\eps(x - y) \, \varphi(y, v) \, dy \, ,
    \end{equation*}
    where the first equality follows since $M^\eps(x-y)=M^\eps(y-x)$, and the dominated convergence theorem applies. A similar argument holds for $j^1$ since $j_t^1 = v \, \sigma_t$ for almost every $t\in[0,T]$.
    
    As for the convergence of the flux $j^2$, using~\eqref{eq:fin-Lag}, we first observe that a finite rate functional gives, for almost every $t \in [0, T]$,
    \begin{align*}
        \langle \varphi,j_t^2\rangle \le \Ent(j_t^2\mid\lambda \, \sigma_t) + \int_{\Omega_V}\lambda \,\bigl(e^{\varphi}-1\bigr) \, d\sigma_t \le \Ent(j_t^2\mid\lambda \, \sigma_t) + \lambda \, (e-1)\,,
    \end{align*}
    for any function $\varphi\in C(\Omega_V)$ with $|\varphi|\le 1$. Taking the supremum over such functions $\varphi$ yields
    \[
        \|j_t^2\|_{\TV} \le \Ent(j_t^2\mid\lambda \, \sigma_t) + \lambda \, (e-1)\qquad\text{for almost every } t\in[0,T] \,.
    \]
    As a consequence, the argument for $\sigma_t^\eps$ holds for $j_t^{2,\eps}$ since the dominated convergence applies.
    
    \UScomment{\\Indeed $\sigma^\eps_t\in \mathcal P(\Omega_V)$, $j^\eps\in \mathcal M((0,T)\times\Omega_V)$ by Fubini-Tonelli and $\int_{\mathbb T_y}M_\eps(x-y)dy=1$ for any $x\in\mathbb T$ and $M_\eps(x-y)=M_\eps(y-x)$. The continuity equation follows since for any $\varphi\in C^{1,0}(\Omega_V)$ we have
    \begin{equation*}
          \langle \varphi, \sigma^\eps_t\rangle - \langle \varphi, \sigma^\eps_s \rangle = \int_s^t \langle \dnabla \varphi, j^\eps_r\rangle\, dr \Longleftrightarrow 
          \langle \varphi^\eps, \sigma_t\rangle - \langle \varphi^\eps, \sigma_s \rangle = \int_s^t \langle \dnabla \varphi^\eps, j_r\rangle\, dr.
    \end{equation*}
    }
    
 \noindent\emph{Ad }(2):
    By construction, $j_r^{1,\eps} = v \, \sigma_t^\eps$. Using Jensen's inequality, for any $\varphi\in C^{1,0}(\Omega_V)$ we find
\begin{align*}
    \langle \varphi,j_t^{2,\eps}\rangle - \int_{\Omega_V}\lambda\,\bigl(e^{\varphi}-1\bigr) \, d\sigma_t^\eps &= \langle \varphi^\eps,j_t^2\rangle - \int_{\Omega_V} \int_\T\lambda\,\bigl(e^{\varphi(x,v)}-1\bigr)\, M_\eps (x-y)\, \sigma(dydv)\, dx \\
    &\le \langle \varphi^\eps,j_t^2\rangle - \int_{\Omega_V} \lambda\,\bigl(e^{\varphi^\eps(y,v)}-1\bigr)\, \sigma(dydv) \le \Ent(j_t^2\mid\lambda\,\sigma_t) \, .
\end{align*}
Therefore, taking the supremum over $\varphi$ yields $\Ent(j_r^{2,\eps}|\lambda\,\sigma_r^\eps) \le \Ent(j_t^2\mid\lambda\,\sigma_r)$ and, as a consequence,
\[
    \scrL(\sigma_t^\eps,j_t^\eps) \le \scrL(\sigma_t,j_t)\qquad\text{for almost every } t\in[0,T] \, .
\]
For the next two points, we first notice that $\sigma_t^\eps \ll \pi$ for $\eps>0$ and every $t\in[0,T]$. Since $j_t^{1,\eps}=v\,\sigma_t^\eps$ and $j_t^{2,\eps}\ll \sigma_t^\eps$ for almost every $t\in[0,T]$, we also have that $j_t^{\eps}\ll \pi$. We introduce
\[
    \varsigma_t^\eps \coloneqq d\sigma_t^\eps/d\pi,\quad w_t^{2,\eps}=dj_t^{2,\eps}/d\pi,\quad t\in[0,T]
\]
and observe that $(t,x,v)\mapsto \varsigma_t^\eps(x,v)$ is continuous on $[0,T]\times \Omega_V$ for any $\eps>0$, and thus bounded.

\noindent\emph{Ad }(3):
Observe that the regularity of $\Mol_\eps$ for $\eps>0$ allows one to obtain the estimate \[
    \int_{\Omega_V}|\partial_x\varsigma_t^\varepsilon| \,d\pi \le C_\eps 
\]
for some constant $C_\eps>0$ with $C_\eps\to\infty$ as $\eps\to 0$. Consequently, for any $\varphi\in C(\Omega_V)$ with $|\varphi|\le 1$,
\begin{align*}
   \left| \langle \varphi, \sigma_t^\eps\rangle - \langle \varphi, \sigma_s^\eps \rangle \right| &= \bigg\lvert\int_s^t \langle \dpartial_v \varphi,  j_r^{2,\eps}\rangle\, dr - \int_s^t \int_{\Omega_V} \varphi\,v\, \partial_x \varsigma_r^\eps\,d\pi\, dr\bigg\rvert \\
   &\le 2 \int_s^t \Ent(j_r^{2,\eps}\mid\lambda\,\sigma_r^\eps)\,dr + 2 \int_s^t \lambda\,\bigl(e^{2}-1\bigr)\,dr + \int_s^t C_\eps V\,dr \eqqcolon \int_s^t g_\eps(r)\,dr \, ,
\end{align*}
where $v\in\{-V,V\}$. Taking the supremum over such $\varphi$ gives
\begin{equation*}
    \lVert\sigma_t^\eps - \sigma_s^\eps\rVert_{\TV}
    \le \int_s^t g_\eps(r)\,dr\qquad \text{with }\, g_\eps\in L^1((0,T)) \, \text{ for all }\, \eps>0 \, .
\end{equation*}
In particular, the curve $t\mapsto\sigma_t^\eps$ is absolutely continuous with respect to the total variation norm.

\noindent\emph{Ad }(4):
From the continuity equation, we deduce that, for any $t\in (0,T)$ and $0<h<T-t$, 
\[
    \varsigma_{t+h}^\eps - \varsigma_t^\eps = -\int_t^{t+h} \bigl(\ddiv_v w_r^{2,\eps} + v\, \partial_x \varsigma_r^\eps \bigr) \, dr\qquad\text{$\pi$-almost everywhere.}
\]
We can then write
\begin{equation}\label{eq:pre-abs-cont}
    \begin{aligned}
    \frac{\calF(\sigma_{t+h}^\eps) - \calF(\sigma_t^\eps)}{h} &= \frac{1}{h}\int_{\Omega_V} \left[\phi(\varsigma_{t+h}^\eps) - \phi(\varsigma_t^\eps)\right] d\pi \\
    &= \frac{1}{h}\int_{\Omega_V} \Bigl(\int_0^1 \phi'\bigl((1-\beta)\varsigma_{t}^\eps + \beta\varsigma_{t+h}^\eps \bigr)\, d\beta\Bigr)(\varsigma_{t+h}^\eps - \varsigma_t^\eps)\,d\pi \\
    &= -\int_{\Omega_V} \Bigl(\int_0^1 \phi'\bigl((1-\beta)\varsigma_t^\eps + \beta\varsigma_{t+h}^\eps \bigr) \, d\beta\Bigr)\left(\frac{1}{h}\int_{t}^{t+h} \bigl(\ddiv_v w_r^{2,\eps} + v\, \partial_x \varsigma_r^\eps \bigr)\,dr \right) d\pi .
    \end{aligned}
\end{equation}
The absolute continuity of $t\mapsto \calF(\sigma_t^\eps)$ follows since $\sup_{t\in[0,T]}\|\varsigma_t^\eps\|_{L^\infty(\Omega_V)}<\infty$, $\phi'\in C([0,\infty))$, and both $\ddiv_v w^{2,\eps}$ and $v\, \partial_x \varsigma^\eps \in L^1\bigl((0,T)\times \Omega_V,\calL_{(0,T)}\otimes \pi\bigr)$, where $\calL_{(0,T)}$ is the Lebesgue measure on the interval $(0,T)$. 

Applying Lebesgue differentiation theorem to the inner integral in~\eqref{eq:pre-abs-cont} we find
\[
    \frac{1}{h}\int_{t}^{t+h}\bigl(\ddiv_v w_r^{2,\eps} + v\, \partial_x \varsigma_r^\eps \bigr)\,dr\; \longrightarrow\; \ddiv_v w_t^{2,\eps} + v\, \partial_x \varsigma_t^\eps\quad\text{in $L^1(\pi)$ for almost every $t\in(0,T)$}.
\]
Therefore passing $h\rightarrow 0$ in~\eqref{eq:pre-abs-cont} leads to  
\begin{align*}
    \frac{d}{dt}\calF(\sigma_t^\eps) &= -\int_{\Omega_V} \phi'(\varsigma_t^\eps) \, \bigl(\ddiv_v w_t^{2,\eps} + v\, \partial_x \varsigma_t^\eps \bigr) \, d\pi \\
    &= \int_{\Omega_V} \dpartial_v\phi'(\varsigma_t^\eps)\,dj_t^{2,\eps} - \int_{\Omega_V} v\, \partial_x \phi(\varsigma_t^\eps)\,d\pi = \int_{\Omega_V} \dpartial_v\phi'(\varsigma_t^\eps)\,dj_t^{2,\eps} ,
\end{align*}
which holds for almost every $t\in(0,T)$, as asserted.
\end{proof}

We now present the proof of Theorem~\ref{thm:FIR}. 
\begin{proof}[Proof of Theorem~\ref{thm:FIR}]
Consider the regularized pair $(\sigma^\eps,j^\eps)$ as in Lemma~\ref{lem:reg}. Since $(\sigma^\eps,j^\eps)\in\CE(0,T;\Omega_V)$ and $\scrI(\sigma^\eps,j^\eps) \leq \scrI(\sigma,j) <\infty$, by the characterisation~\eqref{eq:fin-Lag} of the rate function and Lemma~\ref{lem:reg}, it follows that $j^\eps=(j^{1,\eps},j^{2,\eps})$ with $j^{1,\eps}=v \, \sigma^\eps$.
As in Lemma~\ref{lem:reg}, we set $\varsigma_t^\eps \coloneqq d\sigma_t^\eps/d\pi$, $t\in[0,T]$. We also note that $\varsigma_t^\eps\in L^\infty(\Omega_V,\pi)$ for all $\eps>0$ and all~$t$.

\emph{Step 1.} For a fixed $\delta>0$, consider the map $f_\delta(r) = r\log(r+\delta)$ and its derivative
\[
    f_\delta'(r) = \log(r+\delta) + \frac{r}{r+\delta} \, .
\]
Clearly, $f_\delta\in C^2([0,\infty))$ and, since $\varsigma_t^\eps\in L^\infty(\Omega_V,\pi)$, we have that
\[
    \int_{\Omega_V} f_\delta (\varsigma_t^\eps)\,d\pi \le  \log\bigl(\lVert\varsigma_t^\eps\rVert_{L^\infty}+\delta\bigr) \qquad\text{for all } t\in[0,T] \, .
\]
In particular, the assumptions of Lemma~\ref{lem:reg}(4) are satisfied. It follows that
\begin{align*}
    \frac{1}{2}\frac{d}{dt} \int_{\Omega_V} f_\delta(\varsigma^\eps_t)\,d\pi = \frac{1}{2}\int_{\Omega_V} \dpartial_v f_\delta'(\varsigma_t^\eps) \, dj_t^{2,\eps}\qquad\text{for almost every } t\in(0,T) \, ,
\end{align*}
and therefore, using the characterisation~\eqref{eq:fin-Lag}, we arrive at
\begin{equation}\label{eq:pre-FIR}
    \frac{1}{2}\frac{d}{dt} \int_{\Omega_V} f_\delta(\varsigma_t^\eps)\,d\pi 
    \leq \scrL(\sigma^\eps_t,j^\eps_t) - \lambda\int_{\Omega_V} \Bigl(1- e^{\frac12\dpartial_v\log(\varsigma_t^\eps+\delta)}\Bigr)\,d\sigma^\eps_t + \frac12\int_{\Omega_V} \dpartial_v\Bigl[\frac{\varsigma^\eps_r}{\varsigma^\eps_r + \delta}\Bigr] \, dj_t^{2,\eps}.
\end{equation}
We may
explicitly write out the second term on the right-hand side in~\eqref{eq:pre-FIR} to find
\begin{align*}
    \int_{\Omega_V} \Bigl(1- e^{\frac12\dpartial_v\log(\varsigma_t^{\eps}+\delta)} \Bigr)\,d\sigma^\eps_t 
    &=  \int_{\Omega_V}\Bigl(1- \sqrt{\mfrac{\varsigma_t^{\eps}\circ \iota+\delta}{\varsigma_t^{\eps}+\delta}}\Bigr)\, \varsigma_t^{\eps}\, d\pi \\
    &= \frac{1}{2}\int_{\Omega_V} \Bigl[ \Bigl(1- \sqrt{\mfrac{\varsigma_t^{\eps}\circ \iota+\delta}{\varsigma_t^{\eps}+\delta}}\Bigr)\, \varsigma_t^{\eps} + \Bigl(1- \sqrt{\mfrac{\varsigma_t^{\eps} +\delta}{\varsigma_t^{\eps}\circ \iota+\delta}}\Bigr)\, \varsigma_t^{\eps}\circ \iota \Bigr]\,d\pi \\
    &= \frac{1}{2}\int_{\Omega_V}\Bigl(\varsigma_t^{\eps} - \sqrt{(\varsigma_t^{\eps}+\delta)\,(\varsigma_t^{\eps}\circ \iota+\delta)} \, \frac{\varsigma_t^{\eps}}{\varsigma_t^{\eps}+\delta}\Bigr) \, d\pi \\
    &\quad+ \frac{1}{2}\int_{\Omega_V}\Bigl(\varsigma_t^{\eps}\circ \iota - \sqrt{(\varsigma_t^{\eps}+\delta)\,(\varsigma_t^{\eps}\circ \iota+\delta)} \, \frac{\varsigma_t^{\eps}\circ \iota}{\varsigma_t^{\eps}\circ \iota+\delta}\Bigr) \, d\pi \\
    &= \frac{1}{2}\int_{\Omega_V} \Bigl[\varsigma_t^{\eps} - \sqrt{(\varsigma_t^{\eps}+\delta)\,(\varsigma_t^{\eps}\circ \iota+\delta)} \Bigl(\frac{\varsigma_t^{\eps}}{\varsigma_t^{\eps}+\delta} + \frac{\varsigma_t^{\eps}\circ \iota}{\varsigma_t^{\eps}\circ \iota+\delta}\Bigr) + \varsigma_t^{\eps}\circ \iota\Bigr] \, d\pi \\
    &\ge \frac{1}{2}\int_{\Omega_V} \Bigl[\varsigma_t^{\eps} - 2\sqrt{(\varsigma_t^{\eps}+\delta)\,(\varsigma_t^{\eps}\circ \iota+\delta)} + \varsigma_t^{\eps}\circ \iota\Bigr] \, d\pi\,\\
    &= \frac{1}{2}\int_{\Omega_V} \Bigl( \sqrt{\varsigma_t^{\eps}+\delta} - \sqrt{\varsigma_t^{\eps}\circ \iota+\delta} \Bigr)^2d\pi - \delta = \FI(\sigma^\eps_t+\delta\,\pi\mid\pi)-\delta \, .
\end{align*}
Substituting back into~\eqref{eq:pre-FIR}, integrating in time over $[0,t]$, and since $\scrL\geq 0$, we find 
\begin{multline}\label{eq:preFIR-delta}
   \frac{1}{2}\int_{\Omega_V} f_\delta(\varsigma^\eps_t)\,d\pi + \lambda\int_0^t \FI(\sigma^\eps_r + \delta\,\pi \mid \pi) \, dr \\ 
    \leq \lambda\,\delta + \frac{1}{2}\int_{\Omega_V}f_\delta(\varsigma^\eps_0)\,d\pi + \scrI(\sigma^\eps,j^\eps) + \frac{1}{2}\int_0^t\int_{\Omega_V} \dpartial_v\frac{\varsigma^\eps_r}{\varsigma^\eps_r + \delta} \, dj_r^{2,\eps}\,dr \, .
\end{multline}

\emph{Step 2.} We now pass $\delta\rightarrow 0$ for fixed $\eps>0$. Using $\int_{\Omega_V} \varsigma_0^\eps\, d\pi = \int_{\Omega_V} \varsigma_t^\eps\, d\pi$, we first rewrite~\eqref{eq:preFIR-delta} as 
\begin{multline}\label{eq:preFIR-delta2}
    \frac{1}{2}\int_{\Omega_V} \bigl( \varsigma^\eps_t \log(\varsigma^\eps_t + \delta) - (\varsigma_t^\eps+\delta) + 1 \bigr)\,d\pi + \lambda\int_0^t \FI(\sigma^\eps_r+\delta\,\pi \mid \pi) \, dr  \\ 
    \leq \lambda\,\delta+ \frac{1}{2}\int_{\Omega_V} \bigl( \varsigma^\eps_0 \log(\varsigma^\eps_0 + \delta) - (\varsigma_0^\eps+ \delta) + 1\bigr)\,d\pi + \scrI(\sigma^\eps,j^\eps) + \frac{1}{2}\int_0^t\int_{\Omega_V} \dpartial_v\frac{\varsigma^\eps_r}{\varsigma^\eps_r + \delta} \,d j_r^{2,\eps}\,dr \, .
\end{multline}
Since $j^{2,\eps}\in \calM([0,T]\times\Omega_V)$ and $\lVert j^{2,\eps} \rVert_{\TV([0,T]\times\Omega_V)}<\infty$, we can pass $\delta\rightarrow 0$ in the final term in~\eqref{eq:preFIR-delta2} using the dominated convergence theorem:
\begin{equation*}
    \lim_{\delta\rightarrow 0} \frac{1}{2}\int_0^t\int_{\Omega_V} \dpartial_v\frac{\varsigma^\eps_r}{\varsigma^\eps_r + \delta} \, dj_r^{2,\eps}\,dr = 0 \, .
\end{equation*}
Using $\langle f,\pi\rangle =\int_{\Omega_V} f d\pi$ for convenience, for the second term in the right-hand side of~\eqref{eq:preFIR-delta2} we calculate 
\begin{align*}
    \langle \varsigma^\eps_0 \log(\varsigma^\eps_0 + \delta), \pi\rangle 
    - \langle \varsigma^\eps_0 + \delta, \pi\rangle + \langle 1, \pi\rangle  
    &\leq \langle \varsigma^\eps_0 \log\varsigma^\eps_0, \pi\rangle 
    + \delta - \langle \varsigma^\eps_0 + \delta, \pi\rangle + \langle 1, \pi\rangle \\
    &= \langle \varsigma^\eps_0 \log\varsigma^\eps_0 - \varsigma_0^\eps + 1 , \pi\rangle  = \Ent(\sigma_0^\eps\mid\pi) \, ,
\end{align*}
where the first inequality follows by the concavity of $x\mapsto \log x$. Using $\phi(s)\coloneqq s\log s-s+1$, the first term in the left-hand side of~\eqref{eq:preFIR-delta2} satisfies 
\begin{align*}
\lim_{\delta\rightarrow 0} \bigl(\langle \varsigma^\eps_t \log(\varsigma^\eps_t+\delta), \pi \rangle - \langle  \varsigma^\eps_t+\delta,\pi \rangle + \langle 1,\pi\rangle \bigr) &\geq \liminf_{\delta\rightarrow 0} \langle \phi(\varsigma^\eps_t+\delta),\pi\rangle - \lim_{\delta\rightarrow 0} \langle \delta\log(\varsigma^\eps_t+\delta),\pi \rangle \\
& \geq \langle \phi(\varsigma^\eps_t),\pi\rangle - \lim_{\delta\rightarrow 0} \delta\log(1+\delta)  = \Ent(\sigma^\eps_t\mid\pi) \, ,
\end{align*}
where the second inequality follows by using Fatou's lemma (since $\phi(s)\geq 0$) for the first term and Jensen's inequality applied to the logarithm for the second term. The second term in the left-hand side of \eqref{eq:preFIR-delta2} can be handled similarly with Fatou's lemma. Thereby, passing $\delta\rightarrow 0$ in~\eqref{eq:preFIR-delta2}, we arrive at
\begin{equation}\label{eq:preFIR-eps}
\frac{1}{2}\Ent(\sigma^\eps_t\mid\pi) + \lambda\int_0^t \FI(\sigma^\eps_r\mid\pi) \, dr 
\leq  \frac{1}{2}\Ent(\sigma^\eps_0\mid\pi) + \scrI(\sigma^\eps,j^\eps) \, .
\end{equation}

\emph{Step 3.} We now pass $\eps\rightarrow 0$ in~\eqref{eq:preFIR-eps}. By construction, $\rho_t^\eps$, $j_t^\eps$ converge to $\rho_t$, $j_t$ for all $t\in[0,T]$ with respect to the narrow topology on $\calP(\Omega_V)$ and $\calM(\Omega_V)$ respectively, and therefore 
\begin{align*}
    \Ent(\sigma^\eps_0\mid\pi) &= \int_{\Omega_V} \phi\Bigl( \int_{\T} \Mol_\eps(x-y) \, \varsigma_0(y,v) \, dy \Bigr) \, \pi(dxdv) \\ 
    &\leq \int_{\Omega_V}\int_{\T} \Mol_\eps(x-y) \, \phi\bigl(\varsigma_0(y,v)\bigr) \,dy\, \pi(dxdv) \\
    &= \int_{\Omega_V} \!\!\!\! \underbrace{\Bigl(\int_\T \Mol_\eps(x-y)\,dx\Bigr)}_{=1\;\text{by translational invariance}} \!\!\! \phi\bigl(\varsigma(y,v)\bigr)\, \pi(dydv) = \Ent(\sigma_0\mid\pi) \, ,
\end{align*}
where the inequality follows by Jensen's inequality, which applies since $\phi$ is convex and $\int_{\T} M_\eps = 1$.

    From Lemma~\ref{lem:reg}(2), we have that $\scrL(\sigma_t^\eps,j_t^\eps)\le \scrL(\sigma_t,j_t)$ for all $t\in[0,T]$ and hence
    \[
        \scrI(\sigma^\eps,j^\eps) \le \scrI(\sigma,j) \, .
    \]

Finally, using the weak lower-semicontinuity of the relative entropy and the Fisher information (Lemma~\ref{lem:FI-prop}) and the pointwise-in-time narrow convergence $\sigma_t^\eps\rightharpoonup\sigma_t$ for all $t\in[0,T]$, the final result follows, since 
\[
    \int_0^t \FI(\sigma_r\mid\pi)\,dr \le \int_0^t \liminf_{\eps\to0} \FI(\sigma_r^\eps\mid\pi)\,dr \le \liminf_{\eps\to 0} \int_0^t \FI(\sigma_r^\eps\mid\pi)\,dr \, ,
\]
where the first inequality follows from the weak lower-semicontinuity of $\FI$, and the second one from Fatou's lemma.
\end{proof} 
As we anticipated in Remark~\ref{rem:well-preparedness}, the well-preparedness condition \eqref{well-preparedness} forces the initial datum to be absolutely continuous with respect to the the Lebesgue measure in the $x$-variable. The Kac equation, however, is well-posed for a larger class of initial data, namely for any probability measure on $\Omega_V$. In the following remark, we discuss a generalisation of the FIR inequality without the restrictive assumption on the initial data. 

\begin{rem}\label{rem:genFIR-formal}
Given an arbitrary initial datum $\hat\pi_0\in\calP(\Omega_V)$, let $\hat\pi\in C([0,T];\calP(\Omega_V))$ be the corresponding weak solution to the Kac equation~\eqref{eq:strong-form}. Obviously, $\hat\pi_t$ converges to the stationary solution $\pi$ as $t\rightarrow\infty$. We now provide formal arguments for an FIR inequality to hold for any pair $(\sigma,j)\in\CE(0,T;\Omega_V)$ which satisfies 
\begin{equation*}
    \Ent(\sigma_0\mid\hat\pi_0)+\scrI(\sigma,j)<\infty\, .
\end{equation*}
Here the requirement on the initial data is considerably relaxed since $\hat\pi_0$ need not be absolutely continuous with respect to the stationary measure $\pi$ as required in Theorem~\ref{thm:FIR}. A straightforward consequence is that we can use Dirac measures as initial datum for the Kac equation. 

Assuming densities for all measures involved and following the ideas in the formal arguments before Proposition~\ref{lem:FI-prop}, we find
\begin{align*}
    \frac12\frac{d}{dt} \int_{\Omega_V} \sigma_t\log\frac{\sigma_t}{\hat\pi_t} 
    = \frac12\int_{\Omega_V} \dpartial_v\Bigl(\log\frac{\sigma_t}{\hat\pi_t}\Bigr) \, j^2_t -\frac{\lambda}{2} \int_{\Omega_V} \Bigl(\frac{\sigma_t}{\hat\pi_t}\circ\iota - \frac{\sigma_t}{\hat\pi_t} \Bigr) \, \hat\pi_t \, ,
\end{align*}
where the final term on the right-hand side drops out if $\hat\pi_t$ is replaced by $\pi$. Using the variational form~\eqref{eq:fin-Lag} of the Lagrangian with the choice $\varphi=\frac12\dpartial_v\log\frac{\sigma_t}{\hat\pi_t}$, the above calculation, after integrating in time, leads to the \emph{generalised} FIR inequality 
\begin{equation}\label{genFIR}
    \Ent(\sigma_t\mid\hat\pi_t) + \int_0^T \hat\FI(\sigma_t\mid\hat\pi_t)\, dt \leq \scrI(\sigma,j)+\Ent(\sigma_0\mid\hat\pi_0) \, .
\end{equation}
Comparing this to the FIR inequality~\eqref{eq:FIR-pi}, we note that the stationary solution $\pi$ to the Kac equation has now been replaced by the time-dependent solution $\hat\pi_t$. Consequently,~\eqref{eq:FIR-pi} is a special case of this inequality. 

For densities $\eta\ll\zeta$, the \emph{generalised} Fisher information $\hat\FI$ is defined as
\begin{equation*}
 \hat\FI(\eta\mid\zeta)\coloneqq \int_{\Omega_V} \Bigl\{ \Bigl(\frac{\eta}{\zeta}\circ\iota\Bigr) \, \zeta - \eta -\frac12 \Bigl[ \Bigl(\frac{\eta}{\zeta}\circ\iota\Bigr)^{\!\frac12} (\eta\,\zeta)^{\frac12} - \eta \Bigr] \Bigr\} \, . 
\end{equation*}
This generalised Fisher information is analogous to similar notions introduced for Markov chains in~\cite{HPST20} (specifically, cf.~\cite[\Eq(15b)]{HPST20} with $\lambda=1/2$) and inherits the properties in Proposition~\ref{lem:FI-prop} (see~\cite[Section 2]{HPST20}).   
Similar generalised FIR inequalities also hold for Markov chains~\cite[Theorem 1.6]{HPST20} and stochastic differential equations~\cite[Eq.~(2.55)]{DLPSS18}. 
\end{rem}


\section{Asymptotic limits}\label{sec:Limits}

In this section we make use of the preceding results to study asymptotic limits of the Kac equation, which is equivalent to studying the corresponding limits of the FC system (recall the discussion in Section~\ref{sec:variational}). Specifically, in Section~\ref{subsec:parabolic} we study the parabolic (or diffusive) limit, which corresponds to $V,\lambda\rightarrow\infty$ such that the ratio $V^2/2\lambda$ stays fixed, and in Section~\ref{subsec:hyperbolic} we study the hyperbolic limit, which  corresponds to $\lambda\rightarrow 0$ with a fixed speed $V>0$. For an explanation of these asymptotic limits and the expected limiting dynamics, we refer back to Section~\ref{subsec:main-results}.

The technique that we use in this paper is variational in nature. It consists of proving compactness properties of (approximate) solutions and a liminf inequality for the rate functional.  
In the following, compactness will be established using the Arzel\`a-Ascoli theorem, where the equicontiuity property will make use of the FIR inequality, and specifically the bound on the Fisher information. To prove the liminf inequality, we will use the duality structure of the Lagrangian~\eqref{eq:fin-Lag}; by making educated choices for the test functions in this duality formulation and using the compactness properties, we will construct a limiting functional which characterises both the limiting solution and the fluctuations as $\eps\rightarrow  0$.

The results below are valid for initial data that satisfy \eqref{well-preparedness} and are thus absolutely continuous with respect to the Lebesgue measure in the $x$-variable. An extension to initial conditions in the larger space of probability measures would require the generalized FIR inequality~\eqref{genFIR} and more technical machinery which we skip here.

\subsection{Diffusive limit $V,\lambda\rightarrow\infty$}\label{subsec:parabolic}
As stated above, in the diffusive limit, we consider the limits $V,\lambda\rightarrow \infty$ such that 
\begin{equation*}
\alpha=\frac{V^2}{2\lambda}\quad\text{is a constant.}
\end{equation*}
To achieve this, we rescale the velocity space via
\[
    \mathfrak{r}_V: \Omega_V \to \Omega_1;\quad (x,v) \mapsto \Bigl(x,\frac{v}{V}\Bigr) \, ,
\]
where $\Omega_1\coloneqq \T\times\{-1,1\}$.
Setting $\hat\sigma \coloneqq (\mathfrak{r}_V)_\sharp\sigma \in \calP(\Omega_1)$,  for any $\varphi\in C^{1,0}(\Omega_V)$, we find (recall~\eqref{eq:weak-formulation})
\begin{align*}
    \int_{\Omega_1} \varphi\circ\mathfrak{r}_V^{-1}\,d\hat\sigma_t - \int_{\Omega_1} \varphi \circ\mathfrak{r}_V^{-1}\,d\hat\sigma_s = \int_s^t \int_{\Omega_1} (Q\varphi)\circ\mathfrak{r}_V^{-1}\,d\hat\sigma_r\,dr \, .
\end{align*}
Since $\iota\circ \mathfrak{r}_V^{-1} = \mathfrak{r}_V^{-1}\circ \iota$, it is not difficult to see that the rescaled generator takes the form
\[
    (Q\varphi)\circ \mathfrak{r}_V^{-1} = V\,v\,\partial_x \psi + \frac{V^2}{2\alpha}\bigl( \psi\circ \iota - \psi\bigr) \eqqcolon Q_V\psi\,,\qquad \psi=\varphi\circ\mathfrak{r}_V^{-1} \, .
\]
Moreover, since $\mathfrak{r}_V$ is a smooth diffeomorphism for any $V> 0$, it induces an isomorphism between $\calC^{1,0}(\Omega_V)$ and $\calC^{1,0}(\Omega_1)$. In particular, a weak solution $\sigma^V$ of \eqref{eq:weak-formulation} gives rise to a weak solution $\hat\sigma^V$ of 
\begin{equation}\label{eq:res-hyp-diff}
    \int_{\Omega_1} \varphi\,d\hat\sigma_t - \int_{\Omega_1} \varphi\,d\hat\sigma_s = \int_s^t \int_{\Omega_1} Q_V\varphi\,d\hat\sigma_r\,dr \, ,
\end{equation}
i.e., in strong form, $\hat\sigma\in C([0,T];\calP(\Omega_1))$ solves
\begin{equation*}
\partial_t\hat\sigma + Vv\,\partial_x \hat\sigma = \frac{V^2}{2\alpha}\bigl(\iota_\sharp\hat\sigma-\hat\sigma\bigr) \, .
\end{equation*}
Henceforth, we will use $\sigma$ instead of $\hat\sigma$ for simplicity of notation.

The functional $\scrI^V\colon C([0,T];\calP(\Omega_1))\times \calM([0,T];\calM(\Omega_1))\rightarrow[0,+\infty]$ corresponding to the rescaled equation~\eqref{eq:res-hyp-diff} is 
\begin{equation}\label{eq:rescaled-RF}
\scrI^V(\sigma,j)=\begin{dcases*}
\displaystyle
\int_0^T \scrL^V(\sigma_r,j_r) \, dr \quad &if $(\sigma,j)\in \CE(0,T;\Omega_1)$,\\
+\infty & otherwise,
\end{dcases*}
\end{equation}
where $\CE(0,T;\Omega_1)$ is defined analogously to Definition~\ref{def:CE} and $\scrL^V\colon\calP(\Omega_1)\times\calM(\Omega_1)\rightarrow [0, +\infty]$ is
\begin{equation}\label{eq:fin-Lag-rescaled}
\begin{aligned}
    \scrL^V(\sigma,j)&=
    \begin{dcases*}\Ent\Bigl(j^2\mathrel{\Big\vert}\tfrac{V^2}{2\alpha}\sigma\Bigr) &if $j^1=V\,v\,\sigma$,\\
    +\infty &otherwise,
    \end{dcases*}\\
    &= \begin{dcases*}
        \displaystyle\sup\limits_{\varphi\in C(\Omega_1)} \int_{\Omega_V} \Bigl( \varphi\, dj^2 - \mfrac{V^2}{2\alpha}\bigl(e^{\varphi}-1\bigr) \, d\sigma \Bigr) &if $j^1=V\,v\,\sigma$, \\
        +\infty & otherwise. 
    \end{dcases*} 
\end{aligned}
\end{equation}

\subsubsection{A priori estimates} 
As a preparation for the variational technique, which consists of proving compactness results and a liminf inequality, we need to establish \emph{a priori} estimates for the rescaled system for an arbitrarily fixed $V>0$. These include an FIR inequality for pairs $(\sigma, j) \in \CE(0, T; \Omega_1)$ and a few related results for $\rho$, $\omega$, and a derived flux~$J$.

\begin{theorem}[Rescaled FIR]\label{thm:rescaled-FIR}
 Fix $V>0$ and let $(\sigma,j) \in \CE(0,T;\Omega_1)$ with
 \begin{equation}\label{ass:fixV-bound}
 \Ent(\sigma_0\mid\stat) + \scrI^V(\sigma,j) < \infty \, ,
\end{equation}
where $\sigma|_{t=0}=\sigma_0$. Then, for any $t\in [0,T]$, we have
\begin{equation}\label{eq:rescaled-FIR}
  \Ent(\sigma_t\mid\stat) + \frac{V^2}{2\alpha} \int_0^t \FI(\sigma_r\mid\stat) \, dr \leq \Ent(\sigma_0\mid\stat) + \scrI^V(\sigma,j) \, .
\end{equation}
\end{theorem}
The inequality~\eqref{eq:rescaled-FIR} is the rescaled version of Theorem~\ref{thm:FIR}. Since, in the limit, we expect a diffusion equation for~$\rho$, we now derive a similar estimate that involves the pair $(\rho, \omega)$ at a fixed $V > 0$.

\begin{cor}\label{lem:w-rho-FIR}
Let us make the same assumptions of Theorem~\ref{thm:rescaled-FIR}. Define $\rho\in C([0,T];\calP(\T))$ and $\omega\in \calM((0,T)\times\T)$ as
\begin{equation*}
    \rho_t(dx)\coloneqq\sum_{v\in\{-1,1\}}\! \sigma_t(dx,v) \, , \quad \omega(dtdx)=\omega_t(dx) \, dt\coloneqq V\!\sum_{v\in\{-1,1\}}\!v\,\sigma_t(dx,v) \, dt \, ,
\end{equation*}
and $\calG_t\colon\calM_{\geq 0}((0,T)\times\T)\times\calM((0,T)\times\T)\rightarrow [0,\infty]$ as
\begin{equation*}
    \calG_t(\rho,\omega)\coloneqq
    \begin{cases} \displaystyle\int_0^t\int_{\T}\Bigl\lvert\frac{d\omega}{d\rho}(r,x)\Bigr\rvert^2 \rho(drdx) &\text{ if } \omega\ll \rho,\\
    +\infty &\text{ otherwise,}
    \end{cases}
\end{equation*}
where $\calM_{\geq 0}([0,T]\times\T)$ is the space of non-negative Borel measures and $\rho(dtdx)=\rho_t(dx)\,dt$. 

Then, for any $t\in [0,T]$, we have the bound
\begin{equation}\label{eq:FIR-rho-w}
    \Ent(\rho_t\mid\mathcal L_{\T}) + \frac{1}{2\alpha}\calG_t(\rho,\omega) \leq \Ent(\sigma_t\mid\stat) + \frac{V^2}{2\alpha} \int_0^t \FI(\sigma_r\mid\stat) \, dr \, ,
\end{equation}
where $\mathcal L_{\T}$ is the Lebesgue measure on $\T$. In particular, we find a $V$-independent constant $c>0$ such that
\begin{equation}\label{eq:wTV-rhoTV}
    \sup_{V\geq 1}\|\omega\|_{\TV(B)}\leq c \, \lVert\rho\rVert^{\frac12}_{\TV(B)}\qquad\text{for any Borel set }\, B\in \mathcal B((0,T)\times\T)\, .
\end{equation}
\end{cor}
\begin{proof}
Standard properties of relative entropy imply that 
\begin{equation*}
    \Ent(\rho_t\mid\mathcal L_{\T}) \leq \Ent(\sigma_t\mid\pi) \, .
\end{equation*}
Since $\FI(\cdot\mid\pi)\geq 0$,~\eqref{eq:rescaled-FIR} implies that $\Ent(\sigma_t\mid\pi)<\infty$, i.e., $\sigma_t\ll \pi$ for any $t\in [0,T]$. Using $\varsigma_t\coloneqq\frac{d\sigma_t}{d\pi}$ and $\pi(dx,1)=\pi(dx,-1)=\frac12 dx$, we can rewrite $\rho$ and $\omega$ as
\begin{equation*}
    \rho(dtdx)=\frac12\sum_{v\in\{-1,1\}}\!\varsigma_t(x,v)\,dtdx \, \quad \text{and} \quad \omega(dtdx)=\frac{V}{2}\!\sum_{v\in\{-1,1\}}\!v\,\varsigma_t(x,v)\,dtdx \, .
\end{equation*}
For almost every $t\in [0,T]$, we have
\begin{multline*}
    \Bigl\lvert\frac{d\omega_{t}}{d\rho_{t}}(x)\Bigr\rvert^2\rho_t(dx) = \frac{V^2\,\Bigl\lvert\sum\limits_{v\in\{-1,1\}}\! v\,\varsigma_t(x,v)\Bigr\rvert^2}{2\,\Bigl\lvert\sum\limits_{v\in\{-1,1\}}\! \varsigma_t(x,v)\Bigr\rvert} \, dx = \frac{V^2\,\bigl\lvert\sqrt{\varsigma_t(x,1)}-\sqrt{\varsigma_t(x,-1)}\bigr\rvert^2\, \Bigl\lvert\sum\limits_{v\in\{-1,1\}}\!\! \sqrt{\varsigma_t(x,v)}\Bigr\rvert^2}{2\,\Bigl\lvert\sum\limits_{v\in\{-1,1\}}\! \varsigma_t(x,v)\Bigr\rvert} \, dx \\
    \leq \frac{V^2 \, \bigl\lvert\sqrt{\varsigma_t(x,1)}-\sqrt{\varsigma_t(x,-1)}\bigr\rvert^2\,\Bigl|\sum\limits_{v\in\{-1,1\}}\! \varsigma_t(x,v)\Bigr\rvert}{\Bigl\lvert\sum\limits_{v\in\{-1,1\}}\! \varsigma_t(x,v)\Bigr\rvert} \, dx = V^2 \, \bigl\lvert\sqrt{\varsigma_t(x,1)}-\sqrt{\varsigma_t(x,-1)}\bigr\rvert^2 \, dx \, ,
\end{multline*}
where the inequality follows by Jensen's inequality. Therefore, for any $t\in [0,T]$,
\begin{align*}
    \int_0^t\int_{\T}\Bigl\lvert\frac{d\omega_{r}}{d\rho_{r}}(x)\Bigr\rvert^2 \, \rho_r(dx)\,dr \leq V^2\int_0^t\FI(\sigma_r\mid\pi)\,dr\, ,
\end{align*}
and the required bound~\eqref{eq:FIR-rho-w} then follows. In particular, for $B\in \mathcal B((0,T)\times\T)$, we find
\begin{align*}
    \lVert\omega\rVert_{\TV(B)} = \int_{B} \Bigl\lvert\frac{d\omega}{d\rho}(x,r)\Bigr\rvert \, \rho(dxdr) 
     \leq \Bigl(\int_{B}\Bigl\lvert\frac{d\omega}{d\rho}(x,r)\Bigr\lvert^2 \rho(dxdr)\Bigr)^{\!\!\frac12}\lVert\rho\rVert^{\frac12}_{\TV(B)} \leq c \, \lVert\rho\rVert^{\frac12}_{\TV(B)} \, ,
\end{align*}
where the final inequality follows from~\eqref{ass:fixV-bound}, \eqref{eq:rescaled-FIR}, and \eqref{eq:FIR-rho-w} for $c$ independent of $V$.
\end{proof}

The probability measure~$\sigma$ has been transformed into the pair of measures $(\rho, \omega)$ via the bijection~\eqref{bijection}. From the fluxes $(j^1, j^2)$, we may derive another four fluxes, but only one---the first moment in $v$ of the flux $j^2$---is relevant when the rate functional is finite (recall ~Section~\ref{sec:Var}). Here we present an \emph{a priori} estimate for a rescaled version of such a derived flux. The reason for such a rescaling will become clear in Section~\ref{sec:prop_lim}. The proof of this estimate makes use of the dual formulation of the rate functional.
\begin{lem}\label{lem:J-apriori}
Under the assumptions of Theorem~\ref{thm:rescaled-FIR}, we have that $j^2\ll\sigma$, and we can thus define $J\in \calM((0,T)\times\T)$ as
\begin{equation*}
    J(dtdx)=J_t(dx)\,dt\coloneqq\!\sum_{v\in\{-1,1\}}\! \frac{v}{V}\,j^2_t(dx,v) \, dt \, .
\end{equation*}
For any $B\in\mathcal B((0,T))$, $C\in\mathcal B(\T)$ and $\beta>0$, we have the bound
\begin{equation*}
    \|J\|_{\TV(B\times C)} \leq \frac{c}{\beta}+\frac{c}{2\beta\alpha}(e^\beta-1) \, \max\Bigl\{\lVert\rho\rVert_{\TV(B\times C)}^{\frac12},\lVert\rho\rVert_{\TV(B\times C)}\Bigr\} \, ,
\end{equation*}
where $c>0$ is independent of $V$.
\end{lem}
\begin{proof}
For any $C\in\mathcal B(\T)$, define $\mathbb 1^\theta_C\coloneqq\mathbb 1_C{\ast}M_\theta$, where $M_\theta$ is the heat kernel on $\T$ (cf.~\eqref{eq:heatker-T} for its definition). 
Choosing $\varphi(x,v)=\beta\psi(x)\mathbb 1^\theta_C(x)\frac{v}{V}$ for any $C\in\mathcal B(\T)$, $\psi\in C(\T)$ with $\lVert\psi\rVert_{L^\infty(\T)}\leq 1$ and $\beta>0$ in the rate functional~\eqref{eq:fin-Lag}, we find 
\begin{align*}
    \frac{1}{\beta}\Bigl\lvert&\int_{\Omega_1}\beta\,\mathbb{1}^\theta_C\,\psi\,\frac{v}{V}\,dj^{2}_t\Bigr\rvert 
    \leq \frac1\beta\scrL^V(\sigma_t,j_t) + \frac1\beta\Bigl\lvert\frac{V^2}{2\alpha}\int_{\Omega_1} \Bigl(e^{\beta\,\mathbb{1}^\theta_C\,\psi\,\frac{v}{V}}-1\Bigr)\,d\sigma_t\Bigr\rvert \\
    &\leq  \frac1\beta\scrL^V(\sigma_t,j_t) + \frac1\beta\Bigl\lvert\frac{V^2}{2\alpha}\int_{\Omega_1} \int_{\T} \Mol_\theta(x-y)\, \Bigl(e^{\beta\,\mathbb{1}_C(y)\,\psi(x)\frac{v}{V}}-1\Bigr)\,dy\,\sigma_t(dxdv)\Bigr\rvert \\
    &= \frac1\beta\scrL^V(\sigma_t,j_t) +\frac1\beta \Bigl\lvert\frac{V^2}{2\alpha}\int_{\Omega_1} \mathbb 1^\theta_C \, v \sinh\frac{\beta\,\psi}{V}\,d\sigma_t\Bigr\rvert + \frac1\beta\Bigl\lvert\frac{V^2}{2\alpha}\int_{\Omega_1} \mathbb 1^\theta_C\,\Bigl( \cosh\frac{\beta\,\psi}{V} - 1 \Bigr)\,d\sigma_t\Bigr\rvert\\
    &\leq \frac1\beta\scrL^V(\sigma_t,j_t) + \frac{V}{2\beta\alpha}\sinh\frac{\beta}{V}\Bigl\lvert\int_{\T}\mathbb 1^\theta_C\,d\omega_t\Bigr\rvert + \frac{V^2}{2\beta\alpha}\Bigl(\cosh\frac{\beta}{V} - 1 \Bigr)\Bigl\lvert\int_{\T} \mathbb 1^\theta_C\,d\rho_t\Bigr\rvert\\
    &\leq \frac{1}{\beta}\scrL^V(\sigma_t,j_t) + \frac{1}{2\beta\alpha}\sinh\beta\,\Bigl\lvert\int_{\T}\mathbb 1^\theta_C\,d\omega_t\Bigr\rvert + \frac{1}{2\beta\alpha}(\cosh\beta - 1)\,\Bigl\lvert\int_{\T} \mathbb 1^\theta_C\,d\rho_t\Bigr\rvert \, .
\end{align*}
Using the dominated convergence theorem, we pass $\theta\rightarrow 0$ in the inequality above to arrive at
\begin{equation*}
     \lVert J_t \rVert_{\TV(C)}
     \leq 
    \frac{1}{\beta}\scrL^V(\sigma_t,j_t) + \frac{1}{2\beta\alpha}\sinh\beta\,\lVert\omega_t\rVert_{\TV(C)} + \frac{1}{2\beta\alpha}(\cosh\beta - 1)\,\lVert\rho_t\rVert_{\TV(C)} \, .
\end{equation*}
Therefore, for any $B\in\mathcal B((0,T))$ along with ~\eqref{eq:wTV-rhoTV} and $\lVert\rho\rVert_{\TV(B\times C)}\leq T$, we have 
\begin{align*}
    \int_B\lVert J_t \rVert_{\TV(C)} \, dt &\leq \frac{1}{\beta}\int_B\scrL^V(\sigma_t,j_t)\, dt + \frac{1}{2\beta\alpha}\sinh\beta \, \lVert\omega\rVert_{\TV(B\times C)}+ \frac{1}{2\beta\alpha}(\cosh\beta - 1)\,\lVert\rho\rVert_{\TV(B\times C)} \\
    & \leq \frac{1}{\beta}\scrI^V(\sigma,j) + \frac{c}{2\beta\alpha}\sinh\beta \, \lVert\rho\rVert_{\TV(B\times C)}^{\frac12}  + \frac{1}{2\beta\alpha} (\cosh\beta-1)\, \lVert\rho\rVert_{\TV(B\times C)}  \\
    &\leq \frac{1}{\beta}\scrI^V(\sigma,j) + \frac{c}{2\beta\alpha}(e^{\beta}-1)\, \max\Bigl\{\lVert\rho\rVert_{\TV(B\times C)}^{\frac12},\lVert\rho\rVert_{\TV(B\times C)}\Bigr\} \, ,
\end{align*}
where $c$ is independent of $V\geq 1$. 
\end{proof}

\subsubsection{Compactness}
We now discuss the compactness properties of various objects involved as $V\rightarrow\infty$. Essentially, there are two levels of compactness, a weaker notion for $\omega^V$ and the derived flux~$J^V$, and a stronger notion for the density~$\rho^V$.

\begin{prop}\label{thm:compact}
Let a sequence $(\sigma^V,j^V) \in \CE(0,T;\Omega_1)$ satisfy
\begin{equation}\label{ass:anyV-bound}
\sup_{V\ge 1} \bigl\{\Ent(\sigma^V_0\mid\stat) + \scrI^V(\sigma^V,j^V)\bigr\}\leq C\quad\text{for some constant }\, C>0\,.
\end{equation}
Define $\rho^V,\omega^V,J^V \in \calM([0,T]\times\T)$ as
\begin{equation}\label{def:wV-rhoV}
\begin{gathered}
\rho^V(dtdx)\coloneqq\!\sum_{v\in\{-1,1\}}\!\sigma_t^V(dx,v)\, dt \, , \quad \omega^V(dtdx)\coloneqq V\!\sum_{v\in\{-1,1\}}\! v \, \sigma_t^V(dx,v) \, dt \, , \\ J^V(dtdx)\coloneqq\!\sum_{v\in\{-1,1\}}\! \frac{v}{V} \, j^{2,V}_t(dx,v) \, dt \, .
\end{gathered}
\end{equation}
Then, there exist subsequences (not relabeled) such that
\begin{enumerate}
\item\label{item:conv-strng-rho} 
$\rho^V\rightarrow \bar{\rho}$ in $C([0,T];\calP(\T))$ with respect to the narrow topology in space.
\item $\omega^V\rightarrow \bar \omega$ in $\calM([0,T]\times\T)$ with respect to the narrow topology.
\item $J^V\rightarrow \bar J$ in $\calM((0,T)\times\T)$ with respect to the narrow topology.
\end{enumerate}
Moreover, $\bar\rho_t\ll \calL_\T$ for every $t\in[0,T]$ and both $\bar\omega$ and $\bar J$ have densities in time, i.e., $\bar \omega(dtdx)=\bar \omega_t(dx)\,dt$, $\bar J(dtdx)=\bar J_t(dx)\,dt$, where $\bar\omega_t$ and $\bar{J}_t$ are defined via disintegration.
\end{prop}
\begin{proof}
The narrow convergence of $\omega^V\!(dtdx)=\omega_t^V\!(dx)\, dt$ is implied by~\eqref{eq:wTV-rhoTV}, which gives
\begin{equation*}
    \sup_{V\geq 1}\lVert\omega^V\rVert_{\TV((0,T)\times\T)} < c \, \lVert\rho^V\rVert^{\frac12}_{\TV((0,T)\times\T)} \leq c \, \sqrt{T} \, ,
\end{equation*}
where $c>0$ is independent of $V\geq 1$ and $\lVert\rho_t\rVert_{\TV(\T)}=1$. The narrow convergence of~$J^V$ follows from Lemma~\ref{lem:J-apriori} since, for any $\beta>0$,
\begin{equation*}
    \sup_{V\geq 1} \lVert J^V\rVert_{\TV((0,T)\times \T)}\leq  \frac{c}{\beta}+\frac{c}{2\beta\alpha}(e^\beta-1)\,\max\bigl\{\sqrt{T},T\bigr\} \, ,
\end{equation*}
where $c$ is independent of $V$. Using the same lemma and \cite[Page 181, Corollary~A5]{Valdier90}, it follows that there exists a measurable family $\bar{J}_t \in \calM(\T)$ such that $\bar{J}(dtdx) = \bar{J}_t(dx)\,dt$.

For every $t\in [0,T]$, the sequence $(\rho_t^V)_{V>0}\subset\calP(\T)$ is pre-compact with respect to the narrow topology. Moreover, since $\sigma^V\in C([0,T];\calP(\Omega_1))$, we have that $t\mapsto \rho^V_t\in C([0,T];\calP(\T))$. Therefore, to prove part~\eqref{item:conv-strng-rho}, we will make use of the Arzel\`a-Ascoli theorem to show that $\rho^V\rightarrow \rho$ in $C([0,T];\calP(\T))$ with respect to the uniform topology in time and narrow topology in space.
To prove equicontinuity of $\rho^V$ in $C([0,T];\calP(\T))$, we will show that
\begin{equation}\label{aux-BL}
    \sup_{V\ge 1}\sup_{t\in [0,T-h]} d_{\BL}(\rho^V_{t+h},\rho^V_t) \xrightarrow{h\rightarrow 0} 0 \, ,
\end{equation}
where $d_{\BL}$ is the bounded Lipschitz metric on the space of probability measures (it induces the narrow topology) and is given by
\begin{equation}\label{eq:BL-metric}
    d_{\BL}(\mu,\nu)\coloneqq\sup\bigl\{ \lvert\langle f,\mu\rangle - \langle f,\nu\rangle \rvert\,: f\in \BL(\T), \|f\|\leq 1 \bigr\} \, ,
\end{equation}
where $\BL(\T)=W^{1,\infty}(\T)$ is the space of bounded Lipschitz functions. Since $C^1(\T)$ (with $\|\cdot\|_{C^{1}(\T)}$) is dense in $W^{1,\infty}(\T)$,~\eqref{aux-BL} is equivalent to showing
\begin{equation}\label{aux-check}
    \sup_{V\ge 1}\sup_{t\in [0,T-h]}\sup_{\substack{\psi\in C^1(\T)\\ \lVert\psi\rVert_{C^1(\T)}\leq 1}} \int_{\T}\psi\,\bigl(d\rho^V_{t+h}-d\rho^V_{t}\bigr) \xrightarrow{h\rightarrow 0} 0 \, .
\end{equation}
Since $(\sigma^V,j^V)\in \CE(0,T;\Omega_1)$, making the choice $\varphi(x,v)=\psi(x)$ in the continuity equation and using $j_t^{1,V}=v\,V\,\sigma_t$, we find
\begin{equation*}
  \int_{\T} \psi \, d\rho^V_t - \int_{\T} \psi \, d\rho^V_s = \int_s^t\int_{\T} \partial_x\psi \, d\omega^{V}_r \, dr \, .
\end{equation*}
Repeating the arguments as in Corollary~\ref{lem:w-rho-FIR} with $B = [t, s] \times \T$, we have
\begin{equation*}
    \lVert\omega^V\rVert_{\TV([s, t]\times\T)} \leq c \, \Bigl(\int_s^t\lVert\rho^V_r\rVert_{\TV(\T)}\,dr\Bigr)^{\!\frac12} \leq c \, \sqrt{t-s} \, ,
\end{equation*}
where $c$ is independent of $V$. Note that, as in the case for $\bar J$, the previous estimate and \cite[Page 181, Corollary~A5]{Valdier90} provide a measurable family $\bar\omega_t\in\calM(\T)$ such that $\bar\omega(dtdx)=\bar\omega_t(dx)\,dt$. 

Using the variational formulation of the total-variation norm, for any $\psi$ with $\lVert\psi\rVert_{C^1(\T)}\leq 1$, we have the bound
\begin{equation*}
    \int_{\T}\psi \, \bigl(d\rho^V_{t+h}-d\rho^V_{t}\bigr) \leq \int_{t}^{t+h} \lVert\omega^V_r\rVert_{\TV(\T)} \, dr \leq c \, \sqrt{h} \, .
\end{equation*}
Equicontinuity follows since the right-hand side is independent of $t$ and $V$. Note that this estimate, in particular, implies a uniform $\frac12$-H\"older estimate with respect to the Wasserstein-1 distance. 

Therefore, by the Arzel\`a-Ascoli theorem, $\rho^V\rightarrow \bar{\rho}$ in $C([0,T];\calP(\T))$ with respect to the uniform topology in time and narrow topology in space and, consequently, we have the pointwise convergence $\rho^V_t\rightarrow\bar\rho_t$ in $\calP(\T)$ with respect to the narrow topology for any $t\in[0,T]$. Note that we have used $\bar\rho(dtdx)=\bar\rho_t(dx) \, dt$, which is true since $\rho^V(dtdx)=\rho^V_t(dx) \, dt$, $\rho^V(dtdx)\rightarrow \bar\rho(dtdx)$ and $\rho^V_t(dx)\rightarrow \bar\rho_t(dx)$ for any $t\in [0,T]$. By uniqueness of the limit and disintegration, this implies that $\bar\rho(dtdx)=\bar\rho_t(dx)\,dt$.

Finally, the fact that $\bar\rho_t\ll \calL_\T$ for all $t\in[0,T]$ follows from the induced FIR inequality on $\rho$ in \eqref{eq:FIR-rho-w} and on the narrow lower semicontinuity of the relative entropy.
\end{proof}

\subsubsection{Properties of the limit system}\label{sec:prop_lim}

Recall from Section~\ref{sec:Var} that the sequences $\rho^V$, $\omega^V$, and $J^V$, with a finite rate functional, satisfy the momentum system~\eqref{eq:weak-FC}. In the following two lemmas, we pass to the limit in these objects and show that the limiting pair $(\bar\rho,\bar\omega)$ satisfies a continuity equation and that $\bar J$ is the distributional derivative of $\bar\rho$.

\begin{lem}\label{lem:barOmega-ac}
Under the assumptions of Proposition~\ref{thm:compact}, let $\rho^V\rightarrow\bar\rho$, $\omega^V\rightarrow\bar\omega$ in $\calM((0,T)\times \T)$ with respect to the narrow topology, and $\rho^V_t\rightarrow\bar\rho_t$ in $\calP(\T)$ with respect to the narrow topology for every $t\in [0,T]$. Then, $\bar{\omega} \ll \bar{\rho}$ in $\calM((0,T)\times\T)$, $\bar{\omega}(dtdx) = \bar{\omega}_t(dx) \, dt$, and $\bar\omega_t\ll\bar\rho_t$ in $\calM(\T)$ for almost every $t\in [0,T]$. Furthermore, the pair $(\bar\rho,\bar\omega)\in \CE(0,T;\T)$, namely it solves
\begin{equation*}
    \partial_t\bar\rho+\partial_x\bar\omega=0 \, ,
\end{equation*}
in the sense that, for any $\psi\in C^1(\T)$ and $0\leq s\leq t\leq T$, we have
\begin{equation*}
    \int_\T\psi \, d\bar\rho_t - \int_\T\psi \, d\bar\rho_s = \int_s^t\int_{\T}\partial_x\psi \, d\bar\omega_r \, dr \, .
\end{equation*}
\end{lem}
\begin{proof}
Since $\rho^V\rightarrow\bar\rho$ and $\omega^V\rightarrow\bar{\omega}$ narrowly in $\calM((0,T)\times\T)$, using the lower-semicontinuity of $\calG_T$~\cite[Theorem 2.34]{AmbrosioFuscoPallara00}, we find
\begin{equation*}
    \calG_T(\bar\rho,\bar\omega)\leq \liminf_{V\rightarrow\infty}\calG_T(\rho^V,\omega^V) < \infty \, ,
\end{equation*}
where the second inequality follows from~\eqref{eq:FIR-rho-w} and \eqref{ass:anyV-bound}. Therefore, $\bar\omega\ll \bar\rho$ in $\calM((0,T)\times\T)$.
From Proposition~\ref{thm:compact}, we know that $\bar\rho(dtdx)=\bar\rho_t(dx)\,dt$. Hence, we conclude that $\bar{\omega}(dtdx) = \bar{\omega}_t(dx) \, dt$ and $\bar{\omega}_t \ll \bar{\rho}_t$ for almost every $t\in [0,T]$. 

Choosing $\varphi(x,v)=\psi(x)$ in the continuity equation and using $j^{1,V}=v\,V\,\sigma^V$, we find
\begin{equation*}
    \int_{\T} \psi \, d\rho^V_t - \int_{\T} \psi \, d\rho^V_s = \int_s^t\int_{\T} \partial_x\psi\, d\omega^V_r \, dr \, .
\end{equation*}
Passing $V\rightarrow\infty$, it follows that $(\bar\rho,\bar\omega)\in\CE(0,T;\T)$.
\end{proof}

In Lemma~\ref{lem:barOmega-ac}, we projected the continuity equation for~$\sigma$ to the corresponding continuity equation for the density~$\rho$ and studied the limit $V\to\infty$. In the next lemma, we perform an analogous operation and find a  continuity equation for the flux~$\omega$. In the limit, under the conditions of Proposition~\ref{thm:compact}, we prove that the flux~$\bar{J}$ is in a one-to-one correspondence with a distributional derivative of $\bar u$, where $\bar u$ is Lebesgue density of $\bar{\rho}$.
With a slight abuse of notation, we will often write $D\bar{\rho}$ as the 
distributional derivative of $\bar u$.

\begin{lem}\label{lem:rho-bv}
Under the assumptions of Proposition~\ref{thm:compact}, let $\rho^V\rightarrow\bar\rho$, $\omega^V\rightarrow\bar\omega$, and $J^V\rightarrow\bar J$ in $\calM((0,T)\times \T)$ with respect to the narrow topology, and $\rho^V_t \to \bar{\rho}_t$ for any $t \in [0, T]$ in $\calP(\T)$ with respect to the narrow topology. Then, for any $\psi\in C^1(\T)$, we have
\begin{equation*}
    \int_\T \bigl( \partial_x\psi \, d\bar\rho_t - 2\psi \, d\bar{J}_t\bigr)=0 \qquad\text{ almost every $t\in [0,T]$\,,}
\end{equation*}
i.e.\ $D\bar\rho_t=-2\bar J_t$ for almost every $t\in (0,T)$. 

In particular, the Lebesgue density $\bar u_t=d\bar\rho_t/d\calL_\T\in BV(\T)$ for almost every $t\in(0,T)$, where $BV(\T)$ denotes the space of functions of bounded variation in $\T$.
\end{lem}
\begin{proof}
From Lemma~\ref{lem:time-space-weak} (which holds unchanged in the rescaled situation with $V=1$), we know that, for any $\chi\in C^1_c((0,T))$ and $\varphi\in C^{1,0}(\Omega_1)$, we have
\begin{equation*}
    \int_0^T\int_{\Omega_1}\dot\chi(t) \, \varphi(x,v) \, \sigma^V_t(dxdv) \, dt = -\int_0^T\int_{\Omega_1}\chi(t) \, \dnabla\varphi(x,v)\cdot j^V_t(dxdv) \, dt \, .
\end{equation*}
Using $\varphi(x,v)=-\psi(x) \frac{v}{V}$, $j^{1,V}_t=v \, V \, \sigma^V_t$, and $J^V(dtdx)\coloneqq\sum\limits_v\frac{v}{V}\,j^{2,V}_t(dx,v)\,dt$, we find
\begin{align*}
    -\frac{1}{V^2}\int_0^T \dot\chi(t) \int_{\T} \psi(x) \, \omega^V_t(dx) \, dt = \int_0^T\chi(t)\int_{\T}\Bigl(\partial_x\psi(x) \, \rho^V_t(dx) - 2 \psi(x) \, J^V_t(dx)\Bigr) \, dt \, .
\end{align*}
Passing $V\rightarrow\infty$, we obtain
\begin{equation*}
    \int_0^T\chi(t)\int_{\T}\bigl( \partial_x\psi(x)\,\bar\rho_t(dx) - 2\psi(x)\,\bar J_t(dx)\bigr) \, dt=0 \qquad \forall \chi\in C^1_c((0,T)) \, ,
\end{equation*}
where we have used Proposition~\ref{thm:compact}. Therefore, for almost every $t\in (0,T)$, we have
\begin{equation*}
    \int_{\T} \bigl(\partial_x\psi \, d\bar\rho_t - 2 \psi \, d\bar{J}_t\bigr) = 0 \qquad \forall\psi\in C^{1}(\T) \, .
\end{equation*}
Since $\bar\rho_t = \bar u_t\calL_\T$ and $\bar J_t$ is a finite Radon measure, the previous equality implies $\bar u_t\in BV(\T)$ for almost every $t\in(0,T)$ (cf.~\cite[Definition 3.1]{AmbrosioFuscoPallara00}).
\end{proof}

\subsubsection{Liminf inequality}
We now prove the liminf inequality, which is the final step of the variational technique. As a special case, this inequality implies that that the sequence of solutions to the FC system, which correspond to minimizers of $\scrI^V$, will converge to the minimizers of the limiting functional.

Define the (limiting) functional $\bar{\scrJ} \colon C([0,T];\calP(\T))\times \calM((0,T)\times\T)\rightarrow [0,\infty]$ by
\begin{equation}\label{def:LimFun-diff}
\bar{\scrJ}(\rho, \omega) =
\begin{dcases*}
    \frac{1}{4\alpha} \int_0^T \int_{\T} \Bigl\lvert \alpha\frac{dD\rho}{d\rho}(t, x) + \frac{d\omega}{d\rho}(t, x) \Bigr\rvert^2 \rho_t(dx) \, dt & if $\omega \ll \rho$, $D\rho\ll \rho$,\\[-8pt]
    & $(\rho, \omega) \in \CE(0, T; \T)$, \\
    + \infty & otherwise,
\end{dcases*}
\end{equation}
where $\CE(0, T; \T)$ is defined in Lemma~\ref{lem:rho-bv}. The minimizers of this functional satisfy the limiting projected continuity equations together with the identity $\omega = - \alpha D\rho$ $\rho$-almost everywhere. 
Combining all these relations, we have, for every $\psi \in C^2(\T)$,
\begin{equation}
    \int_{\T} \psi \, d\rho_t - \int_{\T} \psi \, d\rho_s = \alpha \int_0^T \int_{\T} \partial_x^2 \psi \, d\rho_r \, dr \, ,
\end{equation}
which is the weak form of the diffusion equation
\begin{equation}
    \partial_t \rho = \alpha \, \partial_x^2 \rho \, .
\end{equation}

Although seemingly different at first sight, the limiting variational formulation~\eqref{def:LimFun-diff} is closely connected to the widely known Wasserstein gradient-flow structure~\cite{Otto01,AmbrosioGigliSavare08,MPR14} of the diffusion equation, as we sketch in Remark~\ref{rem:Wasserstein}.


\begin{theorem}[Liminf inequality]\label{thm:liminf}
Under the same conditions as in Proposition~\ref{thm:compact}, let $\rho^V\rightarrow\bar\rho$, $\omega^V\rightarrow\bar\omega$, and $J^V\rightarrow\bar J$ in $\calM((0,T)\times \T)$ with respect to the narrow topology, and $\rho^V_t\rightarrow\bar\rho_t$ in $\calP(\T)$ with respect to the narrow topology for every $t\in [0,T]$. Then,
\begin{equation*}
\liminf_{V\rightarrow\infty} \scrI^V(\sigma^V,j^V) \geq \bar{\scrJ}(\bar\rho,\bar\omega) \, .
\end{equation*}
\end{theorem}
\begin{proof}
Choosing $\varphi(x,v)=\psi(x) \frac{v}{V}$ in~\eqref{eq:rescaled-RF}, we find
\begin{equation}\label{eq:limit-aux1}
\scrI^V(\sigma^V,j^V) \geq \int_0^T\int_{\Omega_1} \Bigl( \psi(x) \, \frac{v}{V} \, j^{2,V}_t(dxdv) - \frac{V^2}{2\alpha} \bigl( e^{\psi(x)\frac{v}{V}}-1\bigr) \, \sigma^V_t(dxdv) \Bigr) \, dt \, .
\end{equation}
For almost every $t\in[0,T]$, we have
\begin{align*}
    V^2 \int_{\Omega_1} \bigl( e^{\psi(x) \frac{v}{V}}-1\bigr) \, \sigma^V_t(dxdv) &\leq \int_{\Omega_1} \Bigl[ V \, v \, \psi(x) + \frac12 \psi(x)^2 \Bigr] \, \sigma^V_t(dxdv) + O\Bigl(\frac1V\Bigr) \\
    &= \int_{\T} \Bigl[ \psi(x) \, \omega^V_t(dx) + \frac12 \psi(x)^2 \, \rho^V_t(dx) \Bigr] + O\Bigl(\frac1V\Bigr) \, ,
\end{align*}
where the inequality follows since $\psi\in L^\infty(\T)$.

Substituting back into~\eqref{eq:limit-aux1}, we arrive at
\begin{equation*}
\scrI^V(\sigma^V,j^V) \geq \int_0^T \biggl( \int_{\T} \psi(x) \, \Bigl( J^V_t(dx) - \frac{1}{2\alpha}\omega^V_t(dx)\Bigr) - \frac{1}{4\alpha}\int_{\T} \psi(x)^2 \rho_t^V(dx)\biggr) \, dr  + O\Bigl(\frac{1}{V}\Bigr) \, .
\end{equation*}
Passing $V\rightarrow\infty$, using Lemma~\ref{lem:barOmega-ac} and Proposition~\ref{thm:compact}, we obtain \begin{align}\label{thm:liminf_inequality}
\begin{aligned}
    \liminf_{V\rightarrow\infty} \scrI^V(\sigma^V,j^V) &\geq \int_0^T\int_\T \Bigl[ \psi(x) \, \Bigl( \bar J(dtdx) -\frac{1}{2\alpha}\bar\omega(dtdx)\Bigr) - \frac{1}{4\alpha} \psi(x)^2 \, \bar\rho(dtdx)\Bigr] \\
    &=\frac{1}{2\alpha}\int_0^T\int_\T \Bigl( \psi(x) \, \bigl( 2\alpha d\bar J_t - d\bar\omega_t\bigr) - \frac{1}{2} \psi(x)^2\,d\bar\rho_t \Bigr) \,  dt \, .
    \end{aligned}
\end{align}
Since the left-hand side is finite, we now claim that $\bar J_t\ll \bar\rho_t$ for almost every $t\in(0,T)$. Indeed, should this not be the case, for a fixed Lebesgue point $t\in(0,T)$, we find a pre-compact set $E\subset \T$ with $\bar\rho_t(E)=0$ and $|\bar J_t|(E)>0$. By the Hahn decomposition theorem, $J_t= J_t^+ - J_t^-$, where $J_t^{\pm}$ are nonnegative measures that are mutually singular. Denoting the supports of $J_t^{\pm}$ by $P^\pm$ respectively, and considering the function $\psi_k = k \, (1_{E\cap P^+} - 1_{E\cap P^-})$, $k\ge 1$, from \eqref{thm:liminf_inequality}, via a smoothing argument, we obtain
\[
   \infty> \liminf_{V\rightarrow\infty} \scrI^V(\sigma^V,j^V)\ge  2\alpha k\int_0^T|\bar J_t|(E)\,dt\qquad\text{for all $k\ge 1$}\,,
\]
where we used $\bar\omega_t\ll \bar\rho_t$ (cf.\ Lemma~\ref{lem:barOmega-ac}).
Sending $k\to\infty$, we arrive at a contradiction, thus implying $\bar J_t \ll \bar\rho_t$. Hence, \eqref{thm:liminf_inequality} leads to
\[
    \liminf_{V\rightarrow\infty} \scrI^V(\sigma^V,j^V) \ge \frac{1}{2\alpha}\int_0^T\int_\T \Bigl[ \psi(x) \, \Bigl( 2\alpha \frac{d\bar J_t}{d\bar\rho_t} - \frac{d\bar\omega_t}{d\bar\rho_t}\Bigr) - \frac{1}{2} \psi(x)^2\Bigr] \, d\bar\rho_t  \, dt \, .
\]
Using $D\bar\rho_t=-2J_t$ from Lemma~\ref{lem:rho-bv} and taking the supremum over $\psi\in C(\T)$, we arrive at the required result by Legendre duality.
\end{proof}

\begin{rem}\label{rem:Wasserstein}
The variational structure~\eqref{def:LimFun-diff} resembles the ``density-flux'' version of the well-known Wasserstein gradient-flow structure for diffusion.
To see this, note that since the limiting $\rho\ll \mathcal L_{\T}$ and $D\rho\ll \rho$, it follows that $\rho\in W^{1,1}(\T)$ and therefore we can write $D\rho/d\rho=\partial_x\rho/\rho$. Expanding the square in~\eqref{def:LimFun-diff} yields
\begin{align*}
    2 \bar{\scrJ}(\rho, \omega) &= \frac{\alpha}{2} \int_0^T\int_{\T}|\partial_x \log\rho_t|^2 \rho_t \, dt + \frac{1}{2\alpha}  \int_0^T\int_{\T} \Bigl|\frac{d\omega_t}{d\rho_t}\Bigr|^2\rho_t\, dt + \int_0^T\int_{\T} \partial_x(\log\rho_t)\, \omega_t\, dt \\
    &= \frac{\alpha}{2} \int_0^T\int_{\T}|\partial_x \log\rho_t|^2 \rho_t \, dt + \frac{1}{2\alpha}  \int_0^T\int_{\T} \Bigl|\frac{d\omega_t}{d\rho_t}\Bigr|^2\rho_t\, dt + \Ent(\rho_T\mid\mathcal L_{\T}) -  \Ent(\rho_0\mid\mathcal L_{\T}) \, ,
\end{align*}
where the second equality follows from integration by parts in the final integral and using the continuity equation $\partial_t\rho = -\partial_x \omega$. The right-hand side of the second equality is exactly the Wasserstein ($\Psi-\Psi^*$) formulation of the diffusion equation~\cite{AmbrosioGigliSavare08,MPR14} where the first term is the Fisher information (or quadratic dual dissipation potential), the second term is the metric derivative in the Wasserstein distance, and the final two terms are the entropy difference.   
\end{rem}

\subsection{Hyperbolic limit $\lambda \to 0$}\label{subsec:hyperbolic}
We now intend to study the hyperbolic limit wherein the switching rate~$\lambda\rightarrow 0$, while the speed~$V$ is kept constant in the Kac equation~\eqref{eq:strong-form}. This limit does not require any rescaling and therefore we directly use the rate functional~\eqref{eq:LDRF}. Since the proof strategy is similar to the diffusive limit, here we only outline the proofs.

\begin{prop}[FIR \& Compactness]\label{prop:lambda_compact}
Let a sequence $(\sigma^\lambda,j^\lambda) \in \CE(0,T;\Omega_V)$ satisfy, for a constant $C>0$, the estimate 
\begin{equation*}
\sup_{\lambda>0}\bigl\{\Ent(\sigma^\lambda_0\mid\stat) + \scrI^\lambda(\sigma^\lambda,j^\lambda)\bigr\}\leq C \, .
\end{equation*}
Define $\rho^\lambda, \omega^\lambda, J^\lambda \in \calM([0,T]\times\T)$ as
\begin{equation*}
\begin{gathered}
\rho^\lambda(dtdx)\coloneqq \! \sum\limits_{v \in \{-V, V\}} \! \sigma_t^\lambda(dx, v) \, dt \, , \quad \omega^\lambda(dtdx) \coloneqq \! \sum\limits_{v \in \{-V, V\}} \! v \, \sigma_t^\lambda(dx, v)\, dt \, , \\
J^\lambda(dtdx) \coloneqq \! \sum\limits_{v \in \{-V, V\}} \! v \, j^2(dx, v) \, dt \, .
\end{gathered}
\end{equation*}
For any $t\in[0,T]$, we have the inequalities
\begin{equation}\label{eq:Lambda-FIR}
     \Ent(\rho^\lambda_t\mid\mathcal L_{\T}) + \frac{\lambda}{V^2}\calG_t(\rho^\lambda,\omega^\lambda) \leq \Ent(\sigma^\lambda_t\mid\stat) + \lambda \int_0^t \FI(\sigma^\lambda_r\mid\stat) \, dr \leq \Ent(\sigma^\lambda_0\mid\stat) + \scrI^\lambda(\sigma^\lambda,j^\lambda) \, .
\end{equation}
Furthermore, there exist subsequences (not relabeled) such that
\begin{enumerate}
\item $\rho^\lambda\rightarrow \bar{\rho}$ in $C([0,T];\calP(\T))$ with respect to the narrow topology in space. 

\item $\omega^\lambda\rightarrow \bar{\omega}$ in $C([0,T];\calM(\T))$ with respect to the narrow topology in space.

\item $J^\lambda\rightarrow \bar J$ in $\calM((0,T)\times\T)$ with respect to the narrow topology and $\bar J(dtdx)=\bar J_t(dx) \, dt$,  where $\bar{J}_t$ is defined via disintegration. 

\end{enumerate}
The limit $\bar\omega\in AC([0,T];\calM(\Omega_V))$, where $\calM(\Omega_V)$ is endowed with the bounded-Lipschitz metric, and, for any $\psi\in C^1(\T)$ and $0\leq s<t\leq T$, satisfies 
\begin{equation}\label{eq:bar-w-CE}
    \int_{\T} \psi (d\bar\omega_t-d\bar\omega_s) = \int_s^t\int_{\T} \bigl( V^2\partial_x\psi\, d\bar\rho_r - 2\psi \, d\bar J_r\bigr) \, dr \, .
\end{equation}
In particular, $t\mapsto \bar\omega_t$ is differentiable almost everywhere with the time-derivative given by 
\begin{equation*}
    \partial_t \bar\omega_t  = -V^2D\bar\rho_t - 2\bar J_t \, ,
\end{equation*}
where $D\bar\rho_t$ is the distributional derivative of the distribution $\psi\mapsto \int_\T \psi d\bar\rho_t$.
\end{prop}
\begin{proof}
The second inequality in~\eqref{eq:Lambda-FIR} is proved in Theorem~\ref{thm:FIR} and the first inequality follows as in Corollary~\ref{lem:w-rho-FIR}. The convergence $\rho^\lambda\rightarrow\bar\rho$ in $C([0,T];\calP(\T))$ follows as in Proposition~\ref{thm:compact}. 

For the convergence of $\omega^\lambda$ in $C([0,T];\calM(\T))$, we will make use of the Arzel\`a-Ascoli theorem. To prove equicontinuity, we will show that (see Proposition~\ref{thm:compact} for a discussion of the equivalence of the condition below to the usual equicontinuity)
\begin{equation}\label{eq:lam-w-eqcont}
        \sup_{0<\lambda\leq 1}\sup_{t\in [0,T-h]}\sup_{\substack{\psi\in C^1(\T)\\ \lVert\psi\rVert_{C^1(\T)}\leq 1}} \int_{\T}\psi\,\bigl(d\omega^\lambda_{t+h}-d\omega^\lambda_{t}\bigr) \xrightarrow{h\rightarrow 0} 0 \, .
\end{equation}
Using $\varphi(x,v)=v\,\psi(x)$ with $\psi\in C^1(\T)$ in the continuity equation~\eqref{eq:CE-IntegralForm} along with $j^{1,\lambda}=v\,\sigma^\lambda$, for any $\beta>0$, we find
\begin{align*}
    \int_{\T}&\psi\,\bigl(d\omega^\lambda_{t+h}-d\omega^\lambda_{t}\bigr) = V^2\int_t^{t+h}  \int_{\T}\partial_x\psi(x) \, \rho^\lambda_r(dx) \, dr - 2\int_t^{t+h} \int_{\Omega_V} v\,\psi(x) \, j^{2,\lambda}_r(dx) \, dr \\
    &\leq V^2\|\psi\|_{C^1(\T)} \, h + \frac{1}{\beta}\int_t^{t+h}\Ent(j^{2,\lambda}_r\mid\lambda\,\sigma^\lambda_r) \, dr + \frac{1}{\beta}\int_t^{t+h} \int_{\Omega_V} \bigl( e^{-2\beta\,v\,\psi(x)}-1\bigr) \, \lambda\,\sigma^\lambda_r(dx) \, dr\\
    &\leq V^2\|\psi\|_{C^1(\T)} \, h + \frac{C}{\beta}
    + \frac{\lambda}{\beta}\bigl( e^{2\beta\,V\|\psi\|_{C^1(\T)}}-1\bigr) \,h \, ,
\end{align*}
where the first inequality follows from the variational form of relative entropy and the second inequality follows since the rate functional is bounded. Since $\beta$ is arbitrary, we can choose it to be sufficiently small such that equicontinuity~\eqref{eq:lam-w-eqcont} follows.

Repeating the arguments as in Corollary~\ref{lem:w-rho-FIR}, there exists $c>0$ independent of $\lambda$ such that 
\begin{equation}\label{eq:Lam-w-bound}
    \lVert \omega^\lambda \rVert^2_{\TV([0, T] \times \Omega_V)} < c \lVert \rho^\lambda\rVert_{\TV([0, T] \times \Omega_V)}\,.
\end{equation}
The narrow convergence of $\sigma^\lambda$ and $\rho^\lambda$ follows from Prokhorov's theorem, since $[0, T] \times \Omega_V$ is compact and $\sigma^\lambda_t, \rho^\lambda_t \in \calP(\Omega_V)$ for every $t \in [0, T]$. The narrow convergence of $\omega^\lambda$ follows similarly as a consequence of~\eqref{eq:Lam-w-bound}. Since the rate functional is finite, $j^{1, \lambda}=v\,\sigma^\lambda\rightarrow v\,\bar\sigma$ narrowly.

Now we discuss the convergence of the fluxes $j^{2, \lambda}$ and $J^\lambda$. For almost every $t \in [0, T]$ and $\varphi\in C(\Omega_V)$ with $\lvert \varphi \rvert \le 1$, we find
\begin{align*}
    \langle \varphi,j_t^{2, \lambda}\rangle \le \Ent(j_t^{2, \lambda}\mid\lambda \, \sigma_t^\lambda) + \int_{\Omega_V}\lambda \,\bigl(e^\varphi-1\bigr) \, d\sigma_t^\lambda \le \Ent(j_t^{2, \lambda}\mid\lambda \, \sigma_t^\lambda) + \lambda \, (e-1) \, ,
\end{align*}
and taking the supremum over these functions yields 
\begin{equation*}
    \lVert j^{2, \lambda} \rVert_{\TV([0, T] \times \Omega_V)} \le \scrI^\lambda(\sigma^\lambda, j^\lambda) + \lambda \, T \, (e-1) \leq C\,,
\end{equation*}
since the rate functional is bounded and $\lambda<1$. Therefore, $j^{2,\lambda}$ converges narrowly in $\calM([0,T]\times\T)$. The narrow convergence of $J^\lambda$ to $\bar J$ follows by repeating the arguments above with $\phi(x,v)=v \, \psi(x)$ for any $\psi\in C(\T)$. The absolute continuity $\bar J(dtdx)=\bar J_t(dx) \, dt$ and the convergence of $\rho_t^\lambda\rightarrow\bar\rho_t$ for every $t\in[0,T]$ follow as in Proposition~\ref{thm:compact}.

Finally,~\eqref{eq:bar-w-CE} follows by once again choosing $\varphi(x,v)=v\,\psi(x)$ with $\psi\in C^1(\T)$ in the continuity equation~\eqref{eq:CE-IntegralForm} and passing $\lambda\rightarrow 0$ with the compactness properties presented above. 
\end{proof}

\begin{lem}\label{lem:lambda_CE}
Under the same assumptions of Proposition~\ref{prop:lambda_compact}, let $\rho^\lambda \to \bar{\rho}$ and $\omega^\lambda \to \bar{\omega}$ in $\calM([0, T]\times \T)$ with respect to the narrow topology and $\rho^\lambda_t\rightarrow\bar\rho_t$ in $\calP(\T)$ with respect to the narrow topology for every $t\in [0,T]$. We then find 
\begin{enumerate}
    \item\label{eq:barW-density} $\bar{\omega} \ll \bar{\rho}$ in $\calM((0,T)\times\T)$ with $\bar{\omega}(dtdx) = \bar{\omega}_t(dx) \, dt$,
    \item\label{eq:bar-CE} $(\bar\rho,\bar\omega)\in \CE(0,T;\T)$ in the sense of Lemma~\ref{lem:barOmega-ac}.
\end{enumerate}
\end{lem}
The proof of points \emph{(1)} and \emph{(2)} of Lemma~\ref{lem:lambda_CE} follows as in the proof of Lemma~\ref{lem:barOmega-ac}. 

\medskip

We define $\bar{\scrJ} \colon C([0, T]; \calP(\T)) \times C([0, T] ; \calM(\T))\times \calM((0,T); \calM(\T)) \to [0, +\infty]$ by 
\begin{equation}\label{eq:wave-var-form}
    \bar{\scrJ}(\rho, \omega, J) \coloneqq
    \begin{dcases*}
        0 & if $(\rho, \omega,J) \in \ME([0, T], \T)$ with $J = 0$, \\
        +\infty & otherwise.
    \end{dcases*}
\end{equation}
Therefore, the minimizers are the weak solution to the wave equation
\begin{subequations}
\begin{align}
    \partial_t \rho &= - \partial_x \omega \, , \\
    \partial_t \omega &= - V^2 \partial_x \rho \, .
\end{align}
\end{subequations}
in the sense of Definition~\ref{def:ME} with $J=0$. 

\begin{rem}
The combination of the two equations above formally yields the wave equation in both variables:
\begin{equation*}
    \partial_t^2 \rho = V^2 \partial_x^2 \rho \, , \qquad \partial_t^2 \omega = V^2 \partial_x^2 \omega \, .
\end{equation*}
Also note that $\bar\scrJ$ in~\eqref{eq:wave-var-form} is the $\lambda\rightarrow 0$ limit of the variational formulation $\hat\scrI$~\eqref{IrhoomegaJ} for the FC system.
\end{rem}

 The variational structure~\eqref{eq:wave-var-form} for the hyperbolic limit is substantially different from the analogous structure~\eqref{def:LimFun-diff} for the parabolic limit. The functional~\eqref{eq:wave-var-form} is simply a characteristic function in the sense of convex analysis: the solutions of the wave equation are the only admissible curves---there are no ``approximate'' solutions. This is fully consistent with the interpretation for the limit of the stochastic Kac process as $\lambda \to 0$. In this regime, we expect a fully deterministic dynamics where probabilities are simply rigidly transported along the straight motion of the particles.

\medskip

We now discuss the liminf inequality. 
\begin{theorem}[$\liminf$ inequality]
Under the same conditions as in Proposition~\ref{thm:compact}, let $\rho^\lambda\rightarrow\bar\rho$, $\omega^\lambda\rightarrow\bar\omega$, and $J^\lambda\rightarrow\bar J$ in $\calM((0,T)\times \T)$ with respect to the narrow topology, and $\rho^\lambda_t\rightarrow\bar\rho_t$ in $\calP(\T)$ with respect to the narrow topology for every $t\in [0,T]$. Then,
\begin{equation*}
\liminf_{\lambda\rightarrow0} \scrI^\lambda(\sigma^\lambda,j^\lambda) \geq \bar{\scrJ}(\bar\rho,\bar\omega,\bar J) \, .
\end{equation*}
\end{theorem}
\begin{proof}
Choosing $\varphi(x,v)=\psi(x) \, v$ in~\eqref{eq:fin-Lag}, we find
\begin{align*}
\scrI^\lambda(\sigma^\lambda,j^\lambda) &\geq \int_0^T\int_{\Omega_V} \Bigl( \psi(x) \, v \, j^{2,\lambda}_t(dxdv) - \lambda \bigl( e^{\psi(x)\,v}-1\bigr) \, \sigma^\lambda_t(dxdv) \Bigr) \, dt \\
&= \int_0^T \Bigl( \int_{\T} \psi(x) \, J^\lambda_t(dx) - \int_{\Omega_V} \lambda \bigl( e^{\psi(x)\,v}-1\bigr) \, \sigma^\lambda_t(dxdv) \Bigr) \, dt \, .
\end{align*}
Passing $\lambda\rightarrow0$ and using the compactness results, we obtain
\begin{equation*}
    \liminf_{\lambda\rightarrow\infty} \scrI^\lambda(\sigma^\lambda,j^\lambda) \geq \int_0^T\int_\T \psi(x) \, \bar{J}_t(dx) \, dt \, .
\end{equation*}
Taking the supremum over $\psi\in C(\T)$, we arrive at the required result.    
\end{proof}

\section{Discussion}\label{sec:Dis}

In this article, we have presented a variational structure for the second-order hyperbolic Fourier-Cattaneo (FC) system by using the large deviations of the (stochastic) Kac process, which is a piecewise-deterministic Markov process. The key ingredient is a bijective mapping which links the law of the Kac process to the FC system and is used to construct the aforementioned variational structure. We then use this structure to present appropriate solution concepts and FIR inequality for these systems. Finally, we study the limiting behaviour of these systems in the diffusive and hyperbolic asymptotic regimes. This work is the first study which offers a variational perspective to measure-valued hyperbolic equations by introducing new solution concepts and variational techniques for scale-bridging.

Although we have limited ourselves to the one-dimensional torus $\T$ as the spatial state-space, we expect that all the ideas readily generalise to the unbounded setting of $\R$ with a possible modification—e.g., we may add a spatial confining potential in the Kac equation to ensure tightness. Since we are interested in connections to the hyperbolic FC system, we are limited to the one-dimensional setting (recall the discussion in Section~\ref{sec:Poisson-Kac}). However, the Kac process and the corresponding Kac equation exist in higher dimensions~\cite{GBC17d} and we will explore these systems and corresponding asymptotic limits in future work. 

The variational structure presented in Section~\ref{sec:variational} for the Kac and FC equations are closely related to recent large-deviation-inspired variational formulations~\cite{AdamsDirrPeletierZimmer11,MPR14,PRS21,PRST22} for (possibly nonlinear) systems of the type 
\begin{equation*}
    \partial_t \sigma = \mathrm{div}\, j, \quad j=j(\sigma),
\end{equation*}
where the flux $j$ only depends on $\sigma$. Note that, even though $j$ plays a similar role to the one in this article, at the level of the macroscopic dynamics, it is a dummy variable. This is in stark contrast to systems studied in this paper where the flux~ $\omega$ has an associated evolution equation making the density-flux pair truly a coupled system. This is to be expected since the hyperbolic heat equation is of hyperbolic type with first and second order derivatives in time.

The `passing to the limit' via the variational structure in Section~\ref{sec:Limits}
is closely related to (Gamma-)limits of $(\Psi,\Psi^*)$-type variational formulations for gradient flows~\cite{SandierSerfaty04,ASZ09,MPR14}. This literature, as in our case, crucially uses the duality structure of the variational formulation and typically assumes well-prepared initial data. Our additional assumption of bounded rate functional arises naturally in the context of large deviations, and our results in some sense capture the convergence of typical behaviour and fluctuations around it. The main difference, as opposed to the aforementioned literature, is that the systems studied in this paper are not gradient flows. This is for instance directly seen from the deterministic transport term in the Kac equation.


\medskip
\paragraph{\textbf{Acknowledgements.}} The authors thank Davide Gabrielli, Massimiliano Giona and Michiel Renger for enlightening discussions on the Kac process. The research of AM was funded by the Swiss National Science
Foundation via the Early Postdoc.Mobility fellowship. The work of US is supported by the Alexander von Humboldt foundation. OT acknowledges support from NWO Vidi grant 016.Vidi.189.102 on ``Dynamical-Variational Transport Costs and Application to Variational Evolution''.

\appendix

\section{From path to flux large deviations}\label{sec:LD}

In this section, we motivate the variational structure of the Kac equation \eqref{marginals} introduced in Section~\ref{subsec:main-results} via a formal reformulation of the large-deviation rate function corresponding to the Kac process.

Let us consider $N$ independent copies of the Kac process on the state space~$\Omega_V \coloneqq \T \times \{-V, V\}$ with initial distribution~$\bar{\sigma}$. The single-particle process is a piecewise deterministic Markov process with deterministic drift~$v$ and jump kernel~$\mu(x, v; \cdot) \coloneqq \lambda \, \delta_{(x, -v)}$, i.e., at each jump, which occurs at rate $\lambda$, the position remains fixed, and the velocity is switched. We denote the law of such a process by $\hat{\bbp}_{\bar{\sigma}} \in \calP(D([0, T], \Omega_V))$. As in Section~\ref{sec:Poisson-Kac}, we then construct the empirical process $\bbP^N \colon (D([0, T], \Omega_V))^N \to \calP(D([0, T], \Omega_V))$ as
\begin{equation*}
    \bbP^N\!\bigl(x^1_\cdot, v^1_\cdot, x^2_\cdot, v^2_\cdot, \ldots, x^N_\cdot, v^N_\cdot\bigr) \coloneqq \frac1N \sum\limits_{i=1}^N \delta_{(x^i_\cdot, v^i_\cdot)} \, .
\end{equation*}
Since the particles are independent, by the large of large numbers, $\bbP^N$ converges almost surely to $\hat{\bbp}_{\bar{\sigma}}$ as $N \to \infty$. Here, however, we are not only interested in the most probable behavior of $\bbP^N$ as $N \to \infty$, but also in the atypical deviations from it. We thus want to find a large-deviation principle for the sequence of measure-valued stochastic processes~$\bbP^N$, which we express informally as
\begin{equation*}
    \operatorname{Prob}\bigl(\bbP^N \approx \bbp\bigr) \asymp e^{-N \I(\bbp)} \qquad \text{as } N \to \infty \, .
\end{equation*}
This means that the fluctuations of the random variable~$\bbP^N$ decay with $N$ in an exponential way, and the rate of decay is quantified in terms of the \emph{rate function}~$\I$.

Since the particles are independent, by Sanov's theorem, the empirical process satisfies a large-deviation principle in $\calP(D([0, T], \Omega_V))$ with rate function
\begin{equation}\label{entropy3}
    \I(\bbp) = \Ent(\bbp \mid \widehat{\bbp}_{\bar{\sigma}}) \, ,
\end{equation}
where $\Ent(\cdot | \cdot)$ is the relative entropy on $\calP(D([0, T], \Omega_V)) \times \calP(D([0, T], \Omega_V))$, defined as
\begin{equation}
    \Ent(\bbp \mid \bbr) \coloneqq
    \begin{cases*}
        \bbE_{\bbp}\Bigl[\log\frac{d\bbp}{d\bbr}\Bigr] & if $\bbp \ll \bbr$, \\
        +\infty & otherwise.
    \end{cases*}
\end{equation}
In this expression, $\bbE_{\bbp}$ denotes the expectation value with respect to the probability measure~$\bbp$ and $d\bbp/d\bbr$ is the Radon-Nikodym derivative of $\bbp$ with respect to $\bbr$, which exists whenever $\bbp$ is absolutely continuous with respect to $\bbr$, namely $\bbp \ll \bbr$. We note that the rate function is minimal and zero at $\bbp = \hat{\bbp}_{\bar{\sigma}}$, i.e., at the most probable realization of the empirical process.

Following \cite[Theorem~2.8]{ConfortiLeonard2022}, one obtains an alternative formulation of the relative entropy~\eqref{entropy3} when $\hat{\bbp}_{\bar{\sigma}}$ is the law of a Markov process. Indeed, when $\bbp$ has finite relative entropy with respect to $\hat{\bbp}_{\bar{\sigma}}$, then $\bbp$ is a solution to the martingale problem with drift $\hat\jmath^1_t(x,v)=(v,0)^\top$ and some (time-dependent) jump kernel $\hat\jmath_t^2\ll \mu=\lambda \, \delta_{(x, -v)}$ for every $t \in [0, T]$. In particular, the time marginal flow $t\mapsto\sigma_t \coloneqq (e_t)_\sharp\bbp$ (with $e_t$ being the time evaluation map) satisfies the Kolmogorov forward equation
\begin{equation}\label{KFEhat}
    \partial_t \sigma + v\,\partial_x \sigma = \int_{\Omega_V} \! \hat\jmath^2(x', v',\cdot) \, \sigma(dx'dv') - \sigma \int_{\Omega_V} \! \hat\jmath^2(\cdot,dx'dv')\,.
\end{equation}
Moreover, the relative entropy takes the expression
\begin{equation*}
    \Ent(\bbp \mid \hat{\bbp}_{\bar{\sigma}}) = \bbE_{\sigma_0}\Bigl[\log\frac{d\sigma_0}{d\bar{\sigma}}\Bigr] + \int_0^T \int_{\Omega_V} \bbE_{\hat\jmath^2_t(x, v; \cdot)}\biggl[\log\frac{d\hat\jmath^2_t(x, v; \cdot)}{d\mu(x, v; \cdot)}\biggr] \, \sigma_t(dxdv) \, dt \, .
\end{equation*}
Therefore, one formally obtains the final formulation of the rate function
\begin{equation}\label{RF-2.5}
    \I(\bbp) =
    \begin{dcases*}
        \Ent(\sigma_0 \mid \bar{\sigma}) + \int_0^T \int_{\Omega_V} \Ent\bigl(\,\hat\jmath^2_t(x, v; \cdot) \mid \mu(x, v; \cdot)\bigr) \, \sigma_t(dxdv) \, dt & if \eqref{KFEhat} holds, 
        \\
        \infty & otherwise.
    \end{dcases*}
\end{equation}

This is almost the starting point of the rest of the paper---all that is left is a slight adjustment of the notation. Let us note that, when $\hat\jmath^2 \ll \mu$, we have
\begin{align*}
    \int_{\Omega_V} & \int_{\Omega_V \setminus \{x, v\}} \! \log\frac{d\hat\jmath^2(x,v;\cdot)}{d\mu(x,v;\cdot)}(x', v') \, \hat\jmath^2(x,v;dx'dv')\,\sigma(dxdv) \\
    &= \int_{\Omega_V} \int_{\Omega_V \setminus \{x, v\}} \frac{d\hat\jmath^2(x,v;\cdot)}{d\mu(x,v;\cdot)}(x', v') \, \log\frac{d\hat\jmath^2(x,v;\cdot)}{d\mu(x,v;\cdot)}(x', v') \, \lambda \, \delta_{x, -v}(dx'dv') \, \sigma(dxdv) \\
    &= \int_{\Omega_V} \, \frac{d\hat\jmath^2(x,v;\cdot)}{d\mu(x,v;\cdot)}(x, -v) \, \log\frac{d\hat\jmath^2(x,v;\cdot)}{d\mu(x,v;\cdot)}(x, -v) \, \lambda  \, \sigma(dxdv) \\
    &= \int_\Omega u(x,v) \log u(x,v) \,\lambda \, \sigma(dx\,dv) = \Ent(j^2 \mid \lambda \, \sigma) \, ,
\end{align*}
where we have defined the measure
\begin{equation*}
    j^2 \coloneqq u \, \lambda \, \sigma \in \calM(\Omega_V),\qquad\text{with}\quad u(x, v) \coloneqq \frac{d\hat\jmath^2(x,v;\cdot)}{d\mu(x,v;\cdot)}(x, -v)\,. 
\end{equation*}
Because of the simple form of the jump kernel~$\mu$, we may choose $j^2$ as a flux variable instead of the full jump kernel~$\hat\jmath^2$. Similarly, we define
\begin{equation*}
    j^1 \coloneqq v\, \sigma \in \calM(\Omega_V) \, .
\end{equation*}

In terms of $j = (j^1, j^2)$, the Kolmogorov forward equation becomes
\begin{equation}\label{KFCE}
    \partial_t \sigma + \partial_x j^1 = \iota_\# j^2 - j^2 \qquad \text{with } j^1 = v \, \sigma \, , \tag{KFE}
\end{equation}
and the rate function now reads
\begin{equation}
    \I(\bbp) = \Ent(\sigma_0 \mid \bar{\sigma}) + \scrI(\sigma, j)\,,
\end{equation}
with
\begin{equation}\label{RateFunction}
 \scrI(\sigma, j) \coloneqq
    \begin{dcases*}
         \int_0^T \Ent(j^2_t \mid \lambda \, \sigma_t) \, dt & if \eqref{KFCE} and $j^1 = v \, \sigma$, 
        \\
        \infty & otherwise,
    \end{dcases*}
\end{equation}
which is precisely the functional defined in \eqref{RF}.

\begin{rem}
    Rate functions of the form \eqref{RF-2.5} appear when establishing large-deviations results related to fluctuations of the fraction of time spent in each state of a random system. This is commonly known in the large-deviation community as large deviations at the level 2.5 (cf.\ \cite{BaratoChetrite2015} and a series of papers by Donsker and Varadhan starting with \cite{DonskerVaradhan1975}).
\end{rem}

Another large-deviation principle that is relevant for us involves the invariant measure of the Kac process, namely the uniform distribution $\stat \coloneqq \calL_{\T} \otimes \operatorname{Unif}_{\!\{-V, V\}}$. Specifically, there is a large-deviation principle for the empirical measure
\begin{equation*}
 \Sigma^N \colon \Omega_V^N \to \calP(\Omega_V) \, ,\qquad \Sigma^N(x^1, v^1, x^2, v^2, \ldots, x^N, v^N) \coloneqq \frac{1}{N} \sum\limits_{i=1}^N \delta_{x^i, v^i} \, ,
\end{equation*}
when $\{(x^i, v^i)\}_{i=1,\ldots,N}$ are i.i.d. random variables distributed according to the invariant measure~$\pi$. Again, Sanov's theorem gives the large-deviation principle
\begin{equation}
 \operatorname{Prob}\bigl(\Sigma^N \approx \sigma\bigr) \asymp e^{-N \bbS(\sigma)} \qquad \text{as } N \to \infty
\end{equation}
with rate function
\begin{equation}
 \bbS(\sigma) = \Ent(\sigma \mid \pi) \, .
\end{equation}

\section{Pre-GENERIC structure for the FC system}\label{sec:preGEN}
It turns out the variational structures introduced in this paper for the Kac equation and the FC system induce pre-GENERIC structures \cite{KZMP20} on the respective state spaces. Similar structures, but fully GENERIC and quadratic in nature, have been proposed for equations similar to the FC system \cite[Section~5.4]{PKG18}. In this appendix we focus on the pre-GENERIC structure for the FC system---the one for the Kac equation being completely analogous. In contrast to \cite{KZMP20}, here we give a formulation in terms of a continuity equation for the pair $(\rho, \omega)$. This extends the formulation of gradient structures in continuity-equation format given in \cite{PS22}.

Let us consider the projection of $\CE(0, T; \Omega_V)$ onto $\T$, namely the system~\eqref{CE_proj}, which we recall here
\begin{align}
 \partial_t \rho + \partial_x J_1^1 &= 0 \, , \\
 \partial_t \omega + \partial_x J_2^1 &= - 2 J^2_2 \, .
\end{align}
We want to write it shortly as
\begin{equation*}
 \partial_t(\rho, \omega) + \mdiv J = 0 \, ,
\end{equation*}
and therefore introduce a new notion of a continuity equation.
\begin{defi}[\textbf{P}rojected \textbf{C}ontinuity \textbf{E}quation]
The quadruple $(\rho,\omega,J_1,J_2)\in \PCE(0,T;\T)$ if
\begin{enumerate}
    \item $(\rho,\omega)\in C([0,T];\calP(\T))\times C((0,T);\calM(\T))$ 
    \item $((J_i)_t)_{t\in (0,T)}\subset \calM(\T; \R^2)$, $i=1,2$, are measurable families satisfying
        \[
            \int_0^T \|(J_i)_t\|_{\TV}\,dt <\infty,
        \]
    \item
    for any $\psi\in C^{1}(\T; \R^2)$ and $0 \leq s \leq t \leq T$,
    \begin{align}
     \langle\psi_1, \rho_t\rangle - \langle\psi_1, \rho_s\rangle &= \int_s^t \bigl\langle (\mnabla\psi)_1, (J_1)_r \bigr\rangle \, dr \, , \\
     \langle\psi_2, \omega_t\rangle - \langle\psi_2, \omega_s\rangle &= \int_s^t \bigl\langle (\mnabla\psi)_2, (J_2)_r \bigr\rangle \, dr \, .
    \end{align}
     where $\mnabla\psi \coloneqq \bigl( (\partial_x\psi_1, 0), (\partial_x\psi_2, -2 \psi_2) \bigr)$ and $\langle a, b \rangle \coloneqq \int_{\T} a \cdot db$.
\end{enumerate}
\end{defi}

The projection of the functional~\eqref{eq:fin-Lag} is then given by
\begin{equation}\label{IrhoomegaJ2}
    \widetilde\scrJ(\rho,\omega,J) \coloneqq \begin{dcases*}
        \int_0^T \widetilde{\scrL}(\rho_t, \omega_t, J_t) \, dt \ \ &if $(\rho,\omega,J)\in \text{PCE}(0,T;\T)$,\\
        +\infty &otherwise,
\end{dcases*}
\end{equation}
with
\begin{equation}\label{LrhoomegaJ}
\begin{aligned}
 \widetilde{\scrL}(\rho, \omega, J) \coloneqq&\; \scrL\bigl(\Pi_V^{-1}(\rho, \omega), \Pi_V^{-1}J^1, \Pi_V^{-1}J^2\bigr) \\
 =&
 \begin{dcases*}
  \Ent\bigl(\Pi_V^{-1}J^2 \mid \lambda \, \Pi_V^{-1}(\rho, \omega)\bigr) & if $\Pi_V^{-1}J^1 = v \, \Pi_V^{-1}(\rho, \omega)$, \\
  +\infty & otherwise.
 \end{dcases*}
 \end{aligned}
\end{equation}
and dual
 \begin{align}\label{HrhoomegaJ}
  \widetilde{\scrH}(\rho, \omega, \psi) &= \sup\limits_{J \in \calM(\T)} \bigl(\langle\psi_1, J_1\rangle + \langle\psi_2, J_2\rangle - \widetilde{\scrL}(\rho, \omega, J)\bigr) \nonumber \\
  &= \int_{\T} \Bigl( \psi^1_1 \, d\omega + V^2 \psi^1_2 \, d\rho + \lambda \, e^{\psi^2_1} \bigl( \cosh(V\psi^2_2) - 1 \bigr) \, d\rho + \frac{\lambda}{V} \, e^{\psi^2_1} \sinh(V\psi^2_2) \, d\omega \Bigr) \, .
 \end{align}
 From this functional, one may construct a pre-GENERIC structure, which we now define for our specific case.
 \begin{defi}[Pre-GENERIC structure and flow in continuity-equation format]
 \sloppy{A \emph{pre-GENERIC structure in continuity-equation format} on the state space $Z \coloneqq \calP(\T) \times \calM(\T)$ is a quadruple $(\mnabla, \scrS, B,\scrR)$ where}
 \begin{enumerate}
  \item a gradient operator~$\mnabla \colon C^1(\T; \R^2) \to C(\T; \R^2) \times C(\T; \R^2)$ with the  transpose~$-\mdiv$;
  \item a continuously differentiable function $\scrS \colon Z \to [0, \infty]$, often called the driving function;
  \item a vector field $B \in \calM(\T; \R^2) \times \calM(\T; \R^2)$ that satisfies $\langle \mnabla d\scrS(z), B(z)\rangle = 0$ for all~$z \in Z$;
  \item a dissipation potential~$\scrR \colon \calP(\T) \times \calM(\T) \times C(\T; \R^2) \times C(\T; \R^2) \to [0, \infty]$ such that $\xi \mapsto \scrR(z, \xi)$ is convex, lower semicontinuous and satisfies $\min\scrR(z, \cdot) = \scrR(z, 0) = 0$ for all $z \in Z$.
 \end{enumerate} 
 The \emph{pre-GENERIC flow in continuity-equation format} corresponding to such structure is the evolution equation given by
 \begin{align*}
  \partial_t z + \mdiv J = 0 \qquad \text{and} \qquad J &= B(z) + \partial_\xi\scrR\bigl(z, -\mnabla d\scrS(z)\bigr) \\
 &= B(z) +  \partial_\psi\widetilde{\scrH}(z, 0) \, .
 \end{align*}
\end{defi}

This definition implies that the dynamics generated by~$B$ preserves the driving function~$\scrS$. Furthermore, along the dynamics generated by the dissipation potential~$\scrR$ the driving function is Lyapunov. In this sense, $B$ is the nondissipative part of the evolution, and the rest is the purely dissipative---the driving function usually has the interpretation of a free energy or (minus) thermodynamic entropy.

Let us discuss the three building blocks $\scrS$, $B$, and $\scrR$ one by one. The driving function is again inspired by large deviations \cite{KZMP20} (see \appendixname~\ref{sec:LD}) which yield the function
\begin{subequations}\label{preGEN}
\begin{equation}
 \scrS(\rho, \omega) \coloneqq \Ent\bigl(\Pi_V^{-1}(\rho, \omega) \mid \Pi_V^{-1}(\calL_{\T}, 0)\bigr) \, .
\end{equation}
The nondissipative vector field $B=(B_1,B_2)$, which can be read off from the linear term in the Hamiltonian has the components
\begin{equation}
 B_1(\rho, \omega) \coloneqq \begin{pmatrix} \omega \\ 0 \end{pmatrix} \qquad \text{and} \qquad B_2(\rho, \omega) \coloneqq \begin{pmatrix} V^2 \rho \\ 0 \end{pmatrix} .
\end{equation}
Finally, the dissipation potential may be recovered from the Hamiltonian by the translation \cite[\Eq(39)]{drmR18b}, which gives
\begin{align}
 \scrR(\rho, \omega, \xi) &= 2 \Bigl[ \scrH\Bigl(\rho, \omega, \frac12 \bigl(\xi + \mnabla d\scrS(\rho, \omega)\bigr) \Bigr) - \scrH\Bigl(\rho, \omega, \frac12 \mnabla d\scrS(\rho, \omega) \Bigr) \Bigr] - \bigl\langle \xi, B(\rho, \omega) \bigr\rangle \nonumber \\
 &= 2 \lambda \int_{\T} \Bigl(\frac{d\sigma_-}{d\calL_{\T}} \frac{d\sigma_+}{d\calL_{\T}}\Bigr)^{\!\frac12} e^{V\frac{\xi_1^2}{2}} \bigl(\cosh(V \xi_2^2) - 1\bigr) \, d\calL_{\T} \qquad \text{with} \quad \sigma_\pm \coloneqq \frac12 \Bigl(\rho \pm \frac{\omega}{V}\Bigr) \, .
\end{align}
\end{subequations}
Apart from the additional exponential dependence on~$\xi_1^2$, this is the classical dissipation potential associated with Markov jump processes \cite{PRST22}. The additional dependence plays no role, since the operator~$\mdiv$ does not act on the component~$J_1^2$.

\begin{rem}
 Note that the functional~\eqref{IrhoomegaJ2} is equivalent to~\eqref{IrhoomegaJ} as long as we set $J_1^2 = \omega$ and $J_2^1 = V^2 \rho$ in \eqref{IrhoomegaJ2}, which indeed are the two conditions that arise from the constraint $\Pi_V^{-1}J^1 = v \, \Pi_V^{-1}(\rho, \omega)$ in \eqref{IrhoomegaJ2}. The notion of the Momentum Equation incorporates the two conditions directly in its definition.
\end{rem}

{\small
\bibliographystyle{unsrt}
\bibliography{Bibliography}

\vspace{0.1cm}
\noindent (A.\ Montefusco) Mathematics of Complex Systems, Zuse-Institut Berlin, 14195 Berlin, Germany \\
Email: \href{mailto:montefusco@zib.de}{montefusco@zib.de}\\[0.2em]
\noindent (U.\ Sharma) Fachbereich Mathematik und Informatik, Freie Universit\"at Berlin, Arnimallee 9, 14195 Berlin, Germany \\
Email: \href{mailto:upanshu.sharma@fu-berlin.de}{upanshu.sharma@fu-berlin.de}\\[0.2em]
\noindent (O.\ Tse) Department of Mathematics and Computer Science, Eindhoven University of Technology, 5600 MB Eindhoven, The Netherlands \\
Email: \href{mailto:o.t.c.tse@tue.nl}{o.t.c.tse@tue.nl}
}

\end{document}